\documentclass[11pt,reqno]{amsart}

\usepackage{amsmath, amsthm, amscd, amsfonts, amssymb, graphicx, color}
\usepackage[bookmarksnumbered, colorlinks, plainpages]{hyperref}
\usepackage{mathrsfs}
\usepackage{geometry}
\usepackage{tikz}
\usepackage{tikz-cd}
\usepackage{pgfplots}
\pgfplotsset{compat=1.18}
\usepackage[english]{babel}
\usepackage[T1]{fontenc}
\usepackage[utf8]{inputenc}

\newtheorem{theorem}{Theorem}[section]
\newtheorem{lemma}[theorem]{Lemma}
\newtheorem{proposition}[theorem]{Proposition}
\newtheorem{corollary}[theorem]{Corollary}
\newtheorem{remark}[theorem]{Remark}
\theoremstyle{definition}
\newtheorem{definition}[theorem]{Definition}
\newtheorem{example}[theorem]{Example}
\newtheorem{problem}[theorem]{Question}

\newcommand{\Cstar}{\mathrm{C}^*}

\title[Functoriality and Weyl Groupoids of Ample C*-Diagonal Pairs]{Functoriality and Weyl Groupoids of Ample C*-Diagonal Pairs}
\author{Ali Jabbari}
\address{Department of Applied Mathematics and Informatics, Kyrgyz-Turkish Manas University, Bishkek, Kyrgyzstan}
\email{ali.jabbari@manas.edu.kg\ \&\ jabbari\_al@yahoo.com}

\begin{document}

\begin{abstract}
We develop a functorial framework for ample C$^*$-diagonal pairs and their Weyl groupoids. The main contribution is a partial functoriality theorem in the untwisted setting: from a diagonal-preserving $*$-homomorphism whose restriction to the diagonal is an isomorphism, we construct an open subgroupoid $\mathcal H_\Phi\subseteq\mathcal G_B$ and a continuous, open, injective groupoid homomorphism $\rho_\Phi: \mathcal H_\Phi\to\mathcal G_A$, together with the induced homeomorphism on unit spaces. We prove that this construction is well defined and extends to a contravariant functor into a category of partial groupoid morphisms, thereby giving a concrete partial Weyl-groupoid functor. We record consequences for expectation-compatible ideals, faithful conditional expectations, and quotient reductions, and formulate questions concerning the interaction with dynamical comparison. We also record tensor-product permanence results, including the product description of Weyl groupoids and the relevant diagonal-dimension estimate. Finally, in the untwisted principal case, we prove a reconstruction/rigidity theorem showing that the underlying Weyl groupoid determines the diagonal pair up to isomorphism. Results imported from Bönicke--Li, Li--Liao--Winter, Renault, Kumjian, and Kopsacheilis--Winter are used explicitly as background or supporting permanence results rather than claimed as new contributions.
\end{abstract}
\maketitle

\section{Introduction}

Cartan subalgebras of \(\mathrm{C}^*\)-algebras, introduced by Kumjian and refined by Renault, play a role in the structure and classification of \(\mathrm{C}^*\)-algebras. A unital inclusion \((D \subset A)\) is called a \emph{\(\mathrm{C}^*\)-diagonal pair} if \(D\) is a maximal abelian, regular subalgebra of \(A\) possessing the unique extension property (UEP). Kumjian's reconstruction theorem shows that a \(C^*\)-diagonal reconstructs a twist over a principal étale groupoid \((\mathcal{G},\Sigma)\). In particular, its associated Weyl groupoid is principal. If one drops the UEP and considers a Cartan pair instead, the corresponding groupoid is, in general, only topologically principal. In the ample case, the unit space is totally disconnected.

In recent years, Kopsacheilis and Winter developed a systematic theory of ample \(\mathrm{C}^*\)-diagonal pairs, establishing deep results on dynamical comparison, tracial almost divisibility, and the equivalence between diagonal comparison and the conjunction of dynamical comparison and strict comparison~\cite{KW}. However, the functorial behavior of such pairs under \(*\)-homomorphisms, particularly under morphisms satisfying explicit diagonal, expectation, and normalizer compatibility conditions, appears to be less developed. Understanding this behavior may provide a useful framework for
questions arising in classification, rigidity, and the construction
of exotic diagonals in tensor products and inductive limits.

Very recently, Sibbel and Winter \cite{SibbelWinter2025} constructed an explicit C*-diagonal with Cantor spectrum in the Cuntz algebra \(\mathcal{O}_2\). Their construction proceeds via crossed products by minimal homeomorphisms on Cantor spaces and infinite tensor products, ultimately relying on Kirchberg's \(\mathcal{O}_2\)-absorption theorem. The present work provides a complementary perspective: the functorial framework developed here (specifically, Theorem~\ref{thm:partial-groupoid} and Theorem~\ref{thm:geom-ideals}) offers a systematic framework for organizing questions about such diagonals under homomorphisms and quotient constructions, subject to the hypotheses stated below.

Additionally, Kopsacheilis and Winter \cite{KopsacheilisWinter2026} have constructed, for each \(n \in \{0,1,2,\ldots,\infty\}\), a C*-diagonal in the CAR algebra \(M_{2^\infty}\) with Cantor spectrum and diagonal dimension exactly \(n\), showing in particular that amenability of the Weyl groupoid is not sufficient to characterize AF diagonals. These examples provide natural test cases for the functorial techniques developed in Section~\ref{sec:partial}.

The main original contribution of the paper is the construction of the partial Weyl-groupoid map

\noindent\textbf{Main contribution.} Let \(\Phi:(A,D)\to(B,C)\) be a diagonal-preserving morphism with \(\Phi|_D:D\to C\) an isomorphism. We construct an open subgroupoid \(\mathcal H_\Phi\subseteq\mathcal G_B\) and a continuous, open, injective groupoid homomorphism
\[
\rho_\Phi:\mathcal H_\Phi\longrightarrow\mathcal G_A,
\]
with unit-space map \(h_\Phi:\mathcal G_B^{(0)}\to\mathcal G_A^{(0)}\). The essential point is that only those normalizers in \(B\) which lift through \(\Phi\) admit canonical germs in \(A\); this forces the domain to be the naturally defined open subgroupoid \(\mathcal H_\Phi\). We prove that the resulting assignment is well defined and contravariantly functorial. We regard this partial functoriality as the main structural contribution
of the paper.
\paragraph{Contribution versus known results.}
For clarity, the groupoid reconstruction theorems, the relevant ideal-structure
results, the tensor-product diagonal-dimension estimate, and the CAR examples
used below are established results from the literature. The new contribution
claimed here is the partial Weyl-groupoid construction for the class of
diagonal-preserving morphisms considered in Section~\ref{sec:partial}, together
with its contravariant categorical formulation and the corresponding
composition theorem. The remaining results are presented as supporting
applications, consequences, or questions arising from this construction.

The paper also contains supporting results and applications. In particular, we analyze quotient and ideal constructions in the untwisted setting under the hypotheses needed for reduced quotient identification; record tensor-product behaviour of ample diagonal pairs and the corresponding diagonal-dimension estimate; and give examples illustrating the framework. These parts rely, where appropriate, on established results in the literature, including the ideal-structure theorem of Bönicke--Li, the tensor-product estimate of Li--Liao--Winter, and the reconstruction theories of Renault and Kumjian. The diagonal-comparison results of Kopsacheilis--Winter are used as background and as a source of examples and open questions.

The resulting organization makes a clear distinction between the new partial functorial construction and the known permanence and reconstruction results on which its applications depend.

\begin{remark}\label{rem:scope}
The present paper is foundational in nature. While we establish a functor \(\mathcal{W}: \mathcal{C}_{\mathrm{amp}}^{\mathrm{iso,untw}} \to \mathsf{Grpd}_{\mathrm{part}}\), we do \emph{not} claim that this functor is an equivalence or that it is faithful in the categorical sense, nor that it yields a complete classification. Rather, we develop the categorical and geometric language needed to study diagonal-preserving morphisms systematically. A concrete object-level rigidity application for \emph{untwisted} pairs is given in Section~\ref{sec:rigidity}, where the underlying Weyl groupoid determines the diagonal pair within the stated class. Further applications remain a direction for future work.
\end{remark}

\begin{remark}
Two morphisms $(\mathcal K,\rho,h),\;(\mathcal K',\rho',h'):\mathcal G\longrightarrow\mathcal H$ in \(\mathsf{Grpd}_{\mathrm{part}}\) are equal if and only if $\mathcal K=\mathcal K'$, $\rho=\rho'$, and $h=h'$. Thus the chosen open subgroupoid \(\mathcal K\) is part of the
morphism data; morphisms are not identified merely because their
underlying partial maps agree.
\end{remark}

After recalling necessary preliminaries (Section~\ref{sec:prelim}), we define morphisms of ample \(\mathrm{C}^*\)-diagonal pairs and establish basic transfer properties (Section~\ref{sec:transfer}) and study quotients that are again groupoid \(\mathrm{C}^*\)-algebras. Section~\ref{sec:completion} gives a geometric characterization of quotient diagonals via invariant open sets, under the assumption of untwistedness. Section~\ref{sec:partial} contains the main functoriality result: the construction of \(\mathcal{H}_\Phi\) and \(\rho_\Phi\) in the untwisted setting considered here. Section~\ref{sec:tensor-products} provides examples and applications, Section~\ref{sec:rigidity} proves a reconstruction/rigidity theorem for untwisted principal pairs, and Section~\ref{sec:open} lists open problems.

Throughout, all \(\mathrm{C}^*\)-algebras are separable and unital. Groupoids are second-countable, locally compact, Hausdorff, \'etale, with compact unit space. A \(\mathrm{C}^*\)-diagonal pair \((D \subset A)\) is called \emph{ample} if \(\operatorname{Spec}(D)\) is totally disconnected.

\section{Preliminaries}\label{sec:prelim}

\subsection{Groupoid conventions}
For a second-countable locally compact Hausdorff étale groupoid
$\mathcal G$, we distinguish the following notions.

\begin{itemize}
    \item[(i)] \emph{Topologically principal} means that the set
    \[
    \left\{x\in\mathcal G^{(0)}:
    \mathcal G_x^x=\{x\}\right\}
    \]
    of units with trivial isotropy is dense in $\mathcal G^{(0)}$.

    \item[(ii)] \emph{Effective} means that $\operatorname{Int}(\operatorname{Iso}(\mathcal G))=\mathcal G^{(0)}$,   where
    \[
    \operatorname{Iso}(\mathcal G)
    =
    \{g\in\mathcal G:r(g)=s(g)\}
    \]
    is the isotropy subgroupoid.
\end{itemize}

Every topologically principal Hausdorff étale groupoid is effective.
Conversely, if $\mathcal G$ is second-countable, Hausdorff, and étale,
then effectiveness implies topological principality. Hence, under the
standing hypotheses above, $\mathcal G$  is effective if and only if $\mathcal G$  is topologically principal, see \cite[Proposition~3.6]{Renault}.

For $C^*$-diagonal pairs, Kumjian's reconstruction theorem gives a
twist over a principal groupoid. In particular, the associated Weyl
groupoid is principal, and hence it is both topologically principal
and effective. We nevertheless distinguish the notions of
``principal'', ``topologically principal'', and ``effective''
throughout the paper, since they are distinct notions in general
outside the present second-countable Hausdorff setting.

\begin{remark}\label{rem:principal-effective}
For a general étale groupoid, one should distinguish effectiveness, topological principality, and principality. Effectiveness is the condition $\operatorname{Int}(\operatorname{Iso}(\mathcal G))
=\mathcal G^{(0)}$, whereas topological principality requires density of the units with trivial isotropy, and principality requires $\operatorname{Iso}(\mathcal G)=\mathcal G^{(0)}$.
For second-countable Hausdorff étale groupoids, effectiveness and
topological principality coincide. For Weyl groupoids arising from
$C^*$-diagonal pairs, Kumjian's reconstruction theorem gives the
stronger property of principality. Consequently, such Weyl groupoids
are effective, so results applicable to Hausdorff effective groupoids
may be invoked once all additional hypotheses of the relevant result
are verified.
\end{remark}

\begin{remark}\label{rem:twist-vs-principality}
A crucial distinction is that principality of the underlying groupoid
does not imply triviality of the associated Kumjian--Renault twist.
Indeed, a principal groupoid may admit a nontrivial twist
\[
\mathbb T\longrightarrow\Sigma\longrightarrow\mathcal G.
\]

In that case the associated reduced twisted groupoid $C^*$-algebra $C_r^*(\mathcal G,\Sigma)$ need not be isomorphic to the untwisted algebra $C_r^*(\mathcal G)$.
We call a $C^*$-diagonal pair \emph{untwisted} when its associated
Kumjian--Renault twist is isomorphic to the trivial twist (equivalently,
represents the trivial twist class). In this case, after the
reconstruction identifications,
\[
A\cong C_r^*(\mathcal G),
\qquad
D\cong C_0(\mathcal G^{(0)}).
\]

Thus, principality of the Weyl groupoid and untwistedness are
independent issues.

The results in Sections~\ref{sec:completion} and
\ref{sec:rigidity} that identify the diagonal pair with an untwisted
groupoid $C^*$-algebra therefore require untwistedness explicitly.
By contrast, constructions that depend only on the Weyl groupoid,
rather than on an identification
$A\cong C_r^*(\mathcal G)$, do not require the associated twist to be
trivial, provided the stated hypotheses of those constructions are
satisfied.
\end{remark}
\subsection{Reconstruction tools and their logical dependencies}
\begin{itemize}
    \item \textbf{Kumjian~\cite[Theorem~3.1]{Kumjian}:}
    Every $C^*$-diagonal pair $(D\subseteq A)$ determines a twist
    \[
    \mathbb T\longrightarrow\Sigma\longrightarrow\mathcal G
    \]
    over a principal étale groupoid $\mathcal G$ such that
    \[
    A\cong C_r^*(\mathcal G,\Sigma),
    \qquad
    D\cong C_0(\mathcal G^{(0)}).
    \]
    Conversely, under the corresponding hypotheses, a twist over a
    principal étale groupoid gives rise to a $C^*$-diagonal pair.

    \item \textbf{Renault~\cite[Theorem~4.1]{Renault}:}
    If $\mathcal G$ is a second-countable, locally compact, Hausdorff,
    étale, topologically principal groupoid, and
    \[
    \mathbb T\longrightarrow\Sigma\longrightarrow\mathcal G
    \]
    is a twist, then
    \[
    \left(
    C_0(\mathcal G^{(0)})
    \subseteq
    C_r^*(\mathcal G,\Sigma)
    \right)
    \]
    is a Cartan pair. Conversely, every separable Cartan pair is obtained,
    up to the appropriate isomorphism, from a twisted second-countable
    locally compact Hausdorff étale groupoid. Under this reconstruction, Cartan pairs correspond to topologically
    principal twisted groupoids. In the untwisted setting considered here,
    the additional UEP condition corresponds to the stronger principality
    condition, so a Cartan pair satisfying the UEP is a $C^*$-diagonal.

    \item \textbf{Bönicke--Li
    \cite[Theorem~3.10 and Corollary~3.12]{BL}:}
    Under the hypotheses used in their ideal-structure theorem,
    in particular inner exactness together with the residual intersection
    property (for example, under the standard essential-principality
    hypotheses), closed ideals of the reduced groupoid $C^*$-algebra are
    generated by their intersection with the canonical diagonal. We use
    this stronger ideal-correspondence result only when its hypotheses
    are explicitly stated below. Separately, for an effective Hausdorff
    étale groupoid we use the standard intersection property that every
    nonzero closed ideal of $C_r^*(\mathcal G)$ has nonzero intersection
    with $C_0(\mathcal G^{(0)})$.

    \item A regular abelian subalgebra with the unique extension property
    is maximal abelian; see \cite[Proposition~4]{Kumjian}. We retain
    maximal abelianness explicitly in the definition below in order to
    make the structure of a $C^*$-diagonal transparent.
\end{itemize}

For a unital $C^*$-diagonal pair $(D\subseteq A)$, the spectrum $\operatorname{Spec}(D)$ is compact. Thus, whenever $D\cong C_0(\mathcal G^{(0)})$, the unit space $\mathcal G^{(0)}$ is compact and hence $C_0(\mathcal G^{(0)})=C(\mathcal G^{(0)})$.

\begin{definition}
A \emph{$C^*$-diagonal pair} is a unital inclusion $D\subseteq A$ together with a faithful conditional expectation $E_A:A\to D$ such that
\begin{itemize}
    \item[(i)] $D$ is a maximal abelian subalgebra of $A$;
    \item[(ii)] $D$ is regular in $A$, that is, $C^*(\mathcal N_A(D))=A$,  where
    \[
    \mathcal N_A(D)
    =
    \left\{
        n\in A:
        nDn^*\subseteq D,\quad
        n^*Dn\subseteq D
    \right\};
    \]

    \item[(iii)] $D$ has the \emph{unique extension property (UEP)}:
    every pure state on $D$ has a unique state extension to $A$.
\end{itemize}
\end{definition}

\begin{remark}
For a $C^*$-diagonal, the canonical conditional expectation onto
$D$ is unique. Nevertheless, we retain $E_A$ as part of the structured
object because expectation compatibility is explicitly required of our
morphisms.
\end{remark}

\begin{definition}
An ample $C^*$-diagonal pair is a $C^*$-diagonal pair
$(D\subseteq A)$ for which $\operatorname{Spec}(D)$
is totally disconnected.
\end{definition}

\begin{remark}\label{rem:principal-effective1}
For an étale groupoid $\mathcal G$, we distinguish the following
properties:
\begin{itemize}
    \item[(i)] $\mathcal G$ is \emph{principal} if $\operatorname{Iso}(\mathcal G)=\mathcal G^{(0)}$;
    \item[(i)] $\mathcal G$ is \emph{topologically principal} if    $\left\{x\in\mathcal G^{(0)}:  \mathcal G_x^x=\{x\}  \right\}$
    is dense in $\mathcal G^{(0)}$;
    \item[(i)] $\mathcal G$ is \emph{effective} if $\operatorname{Int}(\operatorname{Iso}(\mathcal G)) =   \mathcal G^{(0)}$.
\end{itemize}

In general,
\[
\text{principal}
\Longrightarrow
\text{topologically principal}
\Longrightarrow
\text{effective}.
\]

For second-countable locally compact Hausdorff étale groupoids, the
standard Baire-category argument gives the equivalence of effectiveness
and topological principality; see, for example,
\cite[Proposition~3.6]{Renault}.

For the $C^*$-diagonal pairs considered in this paper, the reconstruction
theorem gives the stronger conclusion that the associated Weyl groupoid
is principal. Hence it is automatically topologically principal and
effective. We retain the distinction between these notions throughout
the paper.
\end{remark}

If $(D\subseteq A)$ is a separable ample $C^*$-diagonal pair, then
Kumjian's reconstruction theorem \cite[Theorem~3.1]{Kumjian} yields a
twist
\[
\mathbb T\longrightarrow\Sigma\longrightarrow\mathcal G
\]
over a second-countable, locally compact, Hausdorff, ample, principal
\'{e}tale groupoid $\mathcal G$ such that $A\cong C_r^*(\mathcal G,\Sigma)$, and $D\cong C_0(\mathcal G^{(0)})$. Since $D$ is unital, $\mathcal G^{(0)}$ is compact, and therefore $D\cong C(\mathcal G^{(0)})$. In particular,
\[
\text{principal}
\Longrightarrow
\text{topologically principal}
\Longrightarrow
\text{effective}.
\]

\begin{remark}\label{rem:twist-vs-principality1}
Principality of the underlying groupoid does not imply triviality of
the associated twist. A principal groupoid may admit a nontrivial twist
\[
\mathbb T\longrightarrow\Sigma\longrightarrow\mathcal G,
\]
and hence the associated reduced twisted groupoid $C^*$-algebra $C_r^*(\mathcal G,\Sigma)$ need not be isomorphic to the untwisted algebra $C_r^*(\mathcal G)$.

We call a $C^*$-diagonal pair \emph{untwisted} when its associated
Kumjian--Renault twist is trivial (equivalently, cohomologous to the
trivial twist). In this case, after the reconstruction identifications, $A\cong C_r^*(\mathcal G)$, and $D\cong C(\mathcal G^{(0)})$. Thus principality of the Weyl groupoid and triviality of the associated twist are logically distinct properties.

The results in Sections~\ref{sec:completion} and
\ref{sec:rigidity} that identify the diagonal pair with an untwisted
groupoid $C^*$-algebra therefore require the untwistedness assumption
explicitly. Constructions based only on the Weyl groupoid do not, by
themselves, require the twist to be trivial.
\end{remark}

\begin{definition}\label{def:morphism} Let $(A,D,E_A)$ and $(B,C,E_B)$ be ample $C^*$-diagonal pairs. A \emph{morphism} \[ \Phi:(A,D,E_A)\longrightarrow(B,C,E_B) \] is a unital $*$-homomorphism $\Phi:A\to B$ satisfying:
\begin{enumerate} \item[(D)] $\Phi(D)\subseteq C$;
 \item[(E)] $\Phi\circ E_A=E_B\circ\Phi$;
\item[(N)] $\Phi\bigl(\mathcal N_A(D)\bigr) \subseteq \mathcal N_B(C)$,
and $ \Phi\big|_{\mathcal N_A(D)} : \mathcal N_A(D)\longrightarrow\mathcal N_B(C) $ is injective.
 \end{enumerate}
\end{definition}
\begin{remark}
Condition~\textup{(N)} implies that $\ker\Phi\cap\mathcal N_A(D)=\{0\}$,   but the converse need not hold, since $\mathcal N_A(D)$ is not, in general, a linear subspace of $A$. Thus injectivity of $\Phi$ on $\mathcal N_A(D)$ is retained explicitly in the definition.
 \end{remark}
\begin{remark}\label{rem:surjective-diagonal-isomorphism}
Suppose that $\Phi:(A,D,E_A)\longrightarrow(B,C,E_B)$ is a surjective morphism and that $\Phi|_D:D\longrightarrow C$ is an isomorphism. Then $\ker\Phi\cap D=\{0\}$. If the ideal-structure theorem of Lemma~\ref{lem:ideal-structure}
applies to $A$, every closed ideal $I\triangleleft A$ is generated by its
intersection with $D$. Applying this to $I=\ker\Phi$ gives
\[
\ker\Phi
=
\operatorname{Ideal}_A(\ker\Phi\cap D)
=
\operatorname{Ideal}_A(\{0\})
=
\{0\}.
\]

Hence $\Phi$ is injective and, since it is surjective, $\Phi$ is an
isomorphism. Thus, under the hypotheses of the ideal-structure theorem, there is no
non-injective surjective morphism whose restriction to the diagonal is
an isomorphism. Consequently, every genuinely non-injective quotient
morphism must satisfy $\ker\Phi\cap D\neq\{0\}$.
\end{remark}

\begin{remark}\label{rem:condition-N-independence}
Condition~\textup{(N)} is not a consequence of conditions~\textup{(D)} and~\textup{(E)}. Let $X$ be a nonempty compact totally disconnected space, and set $A=C(X)\oplus C(X)$, $D=A$, and $B=C(X)$, $C=B$. Equip these pairs with the faithful conditional expectations $E_A=\operatorname{id}_A$, and $E_B=\operatorname{id}_B$. Both $(A,D,E_A)$ and $(B,C,E_B)$ are ample $C^*$-diagonal pairs.
Indeed, the algebras are commutative, so the diagonal is maximal abelian and regular, while the unique extension property is immediate from $D=A$ and $C=B$. Moreover, $\operatorname{Spec}(A)\cong X\sqcup X$ and $\operatorname{Spec}(B)\cong X$, so both spectra are totally disconnected.

Define $\Phi:A\longrightarrow B,$ by $\Phi(f,g)=f$. Then $\Phi$ is a unital $*$-homomorphism and $\Phi(D)=C$, so \textup{(D)} holds. Furthermore, $\Phi\circ E_A=\Phi=E_B\circ\Phi$, so \textup{(E)} holds. Since $D=A$ is commutative, $\mathcal N_A(D)=A$. In particular, $n=(0,1)\in\mathcal N_A(D)$, and $n\neq0$, whereas $\Phi(n)=0=\Phi(0)$. Thus, $\Phi\big|_{\mathcal N_A(D)}$ is not injective, and condition~\textup{(N)} fails. Therefore,
\[
\textup{(D)}+\textup{(E)}\not\Longrightarrow\textup{(N).}
\]

Hence the normalizer condition is an independent requirement in the definition of a morphism.
\end{remark}

\begin{remark}
For a $C^*$-diagonal, the canonical conditional expectation onto $D$ is
unique; nevertheless, we retain $E_A$ as part of the structured object
because expectation compatibility is explicitly required of our morphisms.
\end{remark}

\begin{proposition}
\label{prop:category}
Ample $C^*$-diagonal pairs, together with the morphisms defined in
Definition~\ref{def:morphism}, form a category
$\mathcal C_{\mathrm{amp}}$.
\end{proposition}

\begin{proof}
Let $\Phi:(A,D_A,E_A)\longrightarrow(B,D_B,E_B)$ and $\Psi:(B,D_B,E_B)\longrightarrow(C,D_C,E_C)$ be morphisms. Since $\Phi$ and $\Psi$ are unital $*$-homomorphisms,
their composition $\Psi\circ\Phi:A\longrightarrow C$ is also a unital $*$-homomorphism.
We verify conditions \textup{(D)}, \textup{(E)}, and \textup{(N)}.

For condition~\textup{(D)}, we have
\[
(\Psi\circ\Phi)(D_A)
=
\Psi\bigl(\Phi(D_A)\bigr)
\subseteq
\Psi(D_B)
\subseteq
D_C.
\]

Thus $\Psi\circ\Phi$ satisfies condition~\textup{(D)}.

For condition~\textup{(E)}, using the compatibility of both
morphisms with the corresponding conditional expectations, we obtain
\[
\begin{aligned}
E_C\circ(\Psi\circ\Phi)
&=(E_C\circ\Psi)\circ\Phi\\
&=(\Psi\circ E_B)\circ\Phi\\
&=\Psi\circ(E_B\circ\Phi)\\
&=\Psi\circ(\Phi\circ E_A)\\
&=(\Psi\circ\Phi)\circ E_A.
\end{aligned}
\]

Hence $\Psi\circ\Phi$ satisfies condition~\textup{(E)}.
It remains to verify condition~\textup{(N)}. Let $n\in\mathcal N_A(D_A)$. By condition~\textup{(N)} for $\Phi$, $\Phi(n)\in\mathcal N_B(D_B)$. Applying condition~\textup{(N)} for $\Psi$ Implies that 
\[
(\Psi\circ\Phi)(n)
=
\Psi(\Phi(n))
\in
\mathcal N_C(D_C).
\]

Therefore $(\Psi\circ\Phi)(\mathcal N_A(D_A))\subseteq \mathcal N_C(D_C)$. We next verify injectivity on the normalizer. Suppose that $n_1,n_2\in\mathcal N_A(D_A)$ and $(\Psi\circ\Phi)(n_1)=(\Psi\circ\Phi)(n_2)$. Since $\Phi(n_1),\Phi(n_2)\in\mathcal N_B(D_B)$ and $\Psi\big|_{\mathcal N_B(D_B)}$ is injective, it follows that $\Phi(n_1)=\Phi(n_2)$. Since $\Phi\big|_{\mathcal N_A(D_A)}$ is injective, we conclude that $n_1=n_2$. Thus $(\Psi\circ\Phi)\big|_{\mathcal N_A(D_A)}$ is injective, and condition~\textup{(N)} holds.

Finally, for every ample $C^*$-diagonal pair $(A,D_A,E_A)$, the identity map $\operatorname{id}_A:A\longrightarrow A$ is a morphism. Indeed, $\operatorname{id}_A(D_A)=D_A$, $E_A\circ\operatorname{id}_A=\operatorname{id}_A\circ E_A$, and $\operatorname{id}_A(\mathcal N_A(D_A))=\mathcal N_A(D_A)$, with the restriction of $\operatorname{id}_A$ to
$\mathcal N_A(D_A)$ being injective. Hence conditions \textup{(D)}, \textup{(E)}, and \textup{(N)} are satisfied.

Associativity of composition is inherited from the associativity of composition of $*$-homomorphisms. Therefore, the ample $C^*$-diagonal pairs together with the morphisms of Definition~\ref{def:morphism} form a category $\mathcal C_{\mathrm{amp}}$.
\end{proof}
\section{Faithful conditional expectations and dynamical comparison}
\label{sec:transfer}

The purpose of this section is to relate the algebraic properties of
$C^*$-diagonal pairs to the structure of their associated groupoids.
We first record a permanence result for faithful conditional
expectations. We then recall dynamical comparison in the normalizer
form and establish the ideal and quotient facts that will be used
later.

\begin{proposition}\label{prop:faithfulness-split}
Let $(A,D,E_A)$, and $(B,C,E_B)$ be $C^*$-diagonal pairs. Suppose that there are $*$-homomorphisms $\Phi:A\longrightarrow B$, and $\Psi:B\longrightarrow A$ such that $\Phi\circ\Psi=\operatorname{id}_B$ and $E_A\circ\Psi=\Psi\circ E_B$. If $E_A$ is faithful, then $E_B$ is faithful.
\end{proposition}

\begin{proof}
Let $b\in B_+$ satisfy $E_B(b)=0$. Since $\Psi$ is a $*$-homomorphism, $\Psi(b)\geq 0.$ The intertwining relation gives 
\[
E_A(\Psi(b))
=
\Psi(E_B(b))
=
0.
\]

Since $E_A$ is faithful, $\Psi(b)=0$. Hence, $b=(\Phi\circ\Psi)(b)=0$. Thus $E_B$ is faithful.
\end{proof}

\begin{definition}\label{def:dyn-comp}
Let $(D\subseteq A)$ be an ample $C^*$-diagonal pair with $\mathcal T(A)\neq\varnothing$. For projections $p,q\in D$, write $p\preceq_N q$ if there exists a normalizer $n\in\mathcal N_A(D)$ such that $n^*n=p$, and $nn^*\leq q$. We say that $(D\subseteq A)$ has \emph{dynamical comparison} if $\tau(p)<\tau(q)$ for every $\tau\in\mathcal T(A)$ implies $p\preceq_N q$. This is the normalizer formulation of dynamical comparison used for ample $C^*$-diagonals.
\end{definition}

\begin{remark}
The relation $p\preceq_N q$ should not be replaced, without additional hypotheses, by the assertion that there is a normalizer $v$ satisfying only $p=vqv^*$. The definition above records the required support relations $n^*n=p$, $nn^*\leq q$, and is the formulation used throughout this paper.
\end{remark}

\begin{lemma}\label{lem:ideal-intersection}
Let $\mathcal G$ be a Hausdorff étale effective groupoid and put $A=C_r^*(\mathcal G)$, and $D=C_0(\mathcal G^{(0)})$. Then every nonzero closed ideal $I\triangleleft A$ satisfies $I\cap D\neq\{0\}$. 
\end{lemma}

\begin{proof}
This is the intersection property for reduced groupoid $C^*$-algebras of effective Hausdorff étale groupoids; see, for example, \cite[Proposition~4.15(1)]{BrixCarlsenSims2024}. Thus, if
$I\neq\{0\}$, then $I\cap C_0(\mathcal G^{(0)})\neq\{0\}$.
\end{proof}

\begin{lemma}\label{lem:ideal-structure}
Let $\mathcal G$ be a second-countable, Hausdorff, ample \`{e}tale groupoid which is inner exact and has the residual intersection property. Put $A=C_r^*(\mathcal G)$, and $D=C_0(\mathcal G^{(0)})$. Then every closed ideal $I\triangleleft A$ is generated by its intersection with $D$, $I=\operatorname{Ideal}_{A}(I\cap D)$. Equivalently, the closed ideals of $A$ are precisely the dynamical ideals $C_r^*(\mathcal G|_U)$, where $U\subseteq\mathcal G^{(0)}$ is open and invariant.
\end{lemma}

\begin{proof}
By Bönicke--Li \cite{BL}, inner exactness together with the residual intersection property implies that every closed ideal of $C_r^*(\mathcal G)$ is dynamical, hence is generated by its intersection with $C_0(\mathcal G^{(0)})$.
\end{proof}

\begin{lemma}\label{lem:quotient-id}
Let $A=C_r^*(\mathcal G)$, and $D=C_0(\mathcal G^{(0)})$, where $\mathcal G$ is an ample Hausdorff étale groupoid. Assume that $\mathcal G$ is inner exact. Let $U\subseteq\mathcal G^{(0)}$ be open and invariant, and put $I=C_r^*(\mathcal G|_U)$. Then there is a canonical isomorphism $A/I\cong C_r^*(\mathcal G|_{\mathcal G^{(0)}\setminus U})$, under which the image of $D$ is $C_0(\mathcal G^{(0)}\setminus U)$. Equivalently,
\[
(D+I)/I
\cong
C_0(\mathcal G^{(0)}\setminus U).
\]
\end{lemma}

\begin{proof}
Put $F=\mathcal G^{(0)}\setminus U$. For the full groupoid \(C^*\)-algebras, restriction to the closed invariant reduction gives the standard exact sequence
\[
0\longrightarrow C^*(\mathcal G|_U)
\longrightarrow C^*(\mathcal G)
\longrightarrow C^*(\mathcal G|_F)
\longrightarrow 0.
\]

Since \(\mathcal G\) is assumed to be inner exact, the corresponding reduced sequence is exact:
\[
0\longrightarrow C_r^*(\mathcal G|_U)
\longrightarrow C_r^*(\mathcal G)
\overset{q}{\longrightarrow}
C_r^*(\mathcal G|_F)
\longrightarrow 0.
\]

Thus, $\ker q=C_r^*(\mathcal G|_U)=I$, and the first isomorphism theorem implies that $A/I\cong C_r^*(\mathcal G|_F)$. On the diagonal, the restriction map $C_0(\mathcal G^{(0)})\longrightarrow C_0(\mathcal G^{(0)}\setminus U)$ has kernel $C_0(U)=I\cap D$. Hence, $D/(I\cap D)\cong C_0(\mathcal G^{(0)}\setminus U)$, which is naturally identified with $(D+I)/I$.
\end{proof}

\begin{lemma}\label{lem:UEP-transfer}
Let $(A,D)$ be a unital $C^*$-diagonal pair and let $\Phi:A\longrightarrow B$ be a surjective unital $*$-homomorphism. Put $C=\Phi(D)$. Assume that $\Phi|_D:D\longrightarrow C$ is an isometric $*$-isomorphism. Then every pure state on $C$ has a unique state extension to $B$.
\end{lemma}

\begin{proof}
Let $\omega$ be a pure state on $C$, and suppose that $\rho_1,\rho_2$ are state extensions of $\omega$ to $B$. Then $\rho_1\circ\Phi$ and $\rho_2\circ\Phi$ are state extensions to $A$ of the pure state $\omega\circ\Phi|_D$ on $D$. Since $(A,D)$ has the unique extension property, $\rho_1\circ\Phi=\rho_2\circ\Phi$. As $\Phi$ is surjective, it follows that $\rho_1=\rho_2$. Thus $C\subseteq B$ has the unique extension property.
\end{proof}

\begin{lemma}
\label{lem:quotient-expectation}
Let $(A,D,E_A)$ be a $C^*$-diagonal pair and let $\Phi:A\longrightarrow B$ be a surjective unital $*$-homomorphism. Put $C=\Phi(D)$. Suppose that there is a conditional expectation
$E_B:B\longrightarrow C$ such that $\Phi\circ E_A=E_B\circ\Phi$. Then
\begin{enumerate}
\item[(i)] If $\ker\Phi\subseteq\ker E_A$, then $E_B$ is uniquely determined by $E_B(\Phi(a))=\Phi(E_A(a))$, for every $a\in A$.
\item[(ii)] If $\ker\Phi\cap D=\{0\}$ and $E_A$ is faithful, then $E_B$ is faithful.
\end{enumerate}
\end{lemma}

\begin{proof}
For (i), suppose that $\Phi(a)=\Phi(b)$. Then
\[
a-b\in\ker\Phi\subseteq\ker E_A,
\]
so $\Phi(E_A(a))=\Phi(E_A(b))$. Hence, $E_B(\Phi(a))=\Phi(E_A(a))$ is well-defined and uniquely determined.

For (ii), let \(b\in B_+\) satisfy $E_B(b)=0$. Choose \(a\in A\) such that $\Phi(a)=b^{1/2}$. Then $\Phi(a^*a)=b$. Hence, 
\[
0=E_B(b)
 =E_B(\Phi(a^*a))
 =\Phi(E_A(a^*a)).
\]

Since $E_A(a^*a)\in D$ and $\ker\Phi\cap D=\{0\}$, we obtain $E_A(a^*a)=0$. Faithfulness of \(E_A\) gives \(a=0\), and hence $b=\Phi(a^*a)=0$. Thus \(E_B\) is faithful.
\end{proof}

\begin{definition}
\label{def:Ciso}
Let $\mathcal C_{\mathrm{amp}}^{\mathrm{iso}}$ be the subcategory of $\mathcal C_{\mathrm{amp}}$ with the same objects, and whose morphisms
$\Phi:(A,D,E_A)\longrightarrow(B,C,E_B)$ are precisely those morphisms satisfying $\Phi|_D:D\longrightarrow C$ is a $C^*$-isomorphism.
\end{definition}

\begin{remark}
\label{rem:surjective-morphism-observation}
Let $A=C_r^*(\mathcal G)$, and $D=C_0(\mathcal G^{(0)})$, where $\mathcal G$ is ample, Hausdorff, and étale. Suppose that $\Phi:A\longrightarrow B$ is a surjective $*$-homomorphism satisfying
$\ker\Phi\cap D=\{0\}$. If the hypotheses of Lemma~\ref{lem:ideal-structure} hold for
$\mathcal G$, then $\ker\Phi$ is generated by
$\ker\Phi\cap D$. Hence $\ker\Phi=0$, and consequently $\Phi$ is an isomorphism.

Thus the conclusion of this remark relies on the inner-exactness and
topological-principality hypotheses of the ideal-correspondence theorem;
it is not a consequence of effectiveness alone.
\end{remark}

\begin{remark}
\label{rem:invariant-reductions}
For open invariant subsets $U\subseteq V$ of $\mathcal G^{(0)}$, put $X_U=\mathcal G^{(0)}\setminus U$, and $X_V=\mathcal G^{(0)}\setminus V$. Then $X_V\subseteq X_U$. There are two natural maps associated with
$X_V\subseteq X_U$: the extension-by-zero inclusion $C_r^*(\mathcal G|_{X_V})\hookrightarrow C_r^*(\mathcal G|_{X_U})$, and, under inner exactness, the quotient/restriction map
$C_r^*(\mathcal G|_{X_U})\longrightarrow C_r^*(\mathcal G|_{X_V})$. The first map preserves the diagonal and the canonical conditional
expectations and maps normalizers into normalizers. The latter is the map
used in Corollary~\ref{cor:invariant-functor}.

We do not regard the extension-by-zero maps as morphisms in
$\mathcal C_{\mathrm{amp}}$ as defined in Section~2, because they are
generally nonunital when $X_V\neq X_U$. Thus no functor into the
unital category $\mathcal C_{\mathrm{amp}}$ is asserted here.
\end{remark}

\section{Quotient diagonals via invariant open sets}\label{sec:completion}

\textbf{Convention for this section.} All $C^*$-diagonal pairs considered here are untwisted. We distinguish carefully between the Cartan and $C^*$-diagonal settings. Cartan reconstruction uses topological principality under the standard hypotheses, whereas the canonical diagonal associated with a principal groupoid gives a $C^*$-diagonal. Principality of the underlying groupoid is logically separate from triviality of a twist.

\begin{lemma}[Germ criterion]\label{lem:germ-criterion}
Let $(A,D)$ be a $C^*$-diagonal pair with Weyl groupoid $\mathcal G$. Let $n,m\in\mathcal N_A(D)$ and let $x\in\operatorname{dom}(n)\cap\operatorname{dom}(m)$. Then
\[
[n,x]=[m,x]\quad\Longleftrightarrow\quad\exists d\in D,\ d(x)\neq0,\ nd=md.
\]
\end{lemma}

\begin{proof}
By the germ equivalence relation defined in Section~2 (see also \cite[Definition~3.5]{Kumjian}), two representatives $(n,x)$ and $(m,x)$ determine the same germ precisely when there exists $d\in D$ with $d(x)\neq0$ and $nd=md$.
\end{proof}

\begin{lemma}[Invariant reduction and reduced quotient]\label{lem:quotient-id-reduction}
Let $\mathcal G$ be an ample Hausdorff étale groupoid, let $U\subseteq\mathcal G^{(0)}$ be open and invariant, and put $F=\mathcal G^{(0)}\setminus U$ and $I=C_r^*(\mathcal G|_U)$. Assume that $\mathcal G$ is inner exact. Then
\[
0\longrightarrow C_r^*(\mathcal G|_U)\longrightarrow C_r^*(\mathcal G)\longrightarrow C_r^*(\mathcal G|_F)\longrightarrow0
\]
is exact, and, with $A=C_r^*(\mathcal G)$ and $D=C_0(\mathcal G^{(0)})$,
\[
A/I\cong C_r^*(\mathcal G|_F),\qquad I\cap D=C_0(U),\qquad D/(I\cap D)\cong C_0(F).
\]

Under these identifications,
\[
(D/(I\cap D)\subseteq A/I)\cong(C_0(F)\subseteq C_r^*(\mathcal G|_F)).
\]
\end{lemma}

\begin{proof}
The full groupoid $C^*$-algebra has the standard exact sequence associated with $U$. By inner exactness, the corresponding reduced sequence is exact; see \cite[Proposition~10.3.2]{Sims2017}. Hence $A/I\cong C_r^*(\mathcal G|_F)$.

Let $E:A\to D$ be the canonical conditional expectation. Its restriction to $I$ has range $C_0(U)$. If $d\in I\cap D$, then $d=E(d)\in C_0(U)$, while the reverse inclusion is immediate. Thus $I\cap D=C_0(U)$. Since $F$ is closed, restriction of functions gives $D/C_0(U)\cong C_0(F)$, proving the stated identifications.
\end{proof}

\begin{remark}
The inner-exactness hypothesis is essential for the reduced quotient identification. We do not assert reduced exactness for arbitrary étale groupoids without this hypothesis.
\end{remark}

\begin{theorem}[Cartan characterization of the reduction]
\label{thm:geom-ideals}
Let $\mathcal G$ be a second-countable, ample Hausdorff \'{e}tale groupoid that is
inner exact, and let $U\subseteq\mathcal G^{(0)}$ be an open invariant
subset. Put $F=\mathcal G^{(0)}\setminus U$. Then $\mathcal G|_F$  is topologically principal if and only if $C_0(F)\subseteq C_r^*(\mathcal G|_F)$ is a Cartan pair.
Moreover, if $A=C_r^*(\mathcal G)$, $D=C_0(\mathcal G^{(0)})$, and $I=C_r^*(\mathcal G|_U)$, then inner exactness implies that
\[
A/I\cong C_r^*(\mathcal G|_F),
\qquad
D/(I\cap D)\cong C_0(F),
\]
and hence $\mathcal G|_F$  is topologically principal if and only if $D/(I\cap D)\subseteq A/I$  is a Cartan pair.
\end{theorem}

\begin{proof}
Since $\mathcal G$ is second-countable, locally compact, Hausdorff,
\'{e}tale, and ample, the reduction $\mathcal G|_F$ is again
second-countable, locally compact, Hausdorff, \'{e}tale, and ample.
Put $F=\mathcal G^{(0)}\setminus U$. By
Renault's Cartan reconstruction theorem
\cite[Theorem~4.1]{Renault}, we have $\mathcal G|_F$  is topologically principal if and only if $C_0(F)\subseteq C_r^*(\mathcal G|_F)$  is a Cartan pair. Now let $A=C_r^*(\mathcal G)$, $D=C_0(\mathcal G^{(0)})$, and $I=C_r^*(\mathcal G|_U)$. By Lemma~\ref{lem:quotient-id-reduction}, using the inner-exactness hypothesis, there is a canonical isomorphism $A/I\cong C_r^*(\mathcal G|_F)$, under which $D/(I\cap D)\cong C_0(F)$. Therefore $D/(I\cap D)\subseteq A/I$  is Cartan if and only if $C_0(F)\subseteq C_r^*(\mathcal G|_F)$   is Cartan. Combining this equivalence with Renault's theorem proves $\mathcal G|_F$   is topologically principal if and only if $D/(I\cap D)\subseteq A/I$ is Cartan.
\end{proof}

\begin{theorem}\label{thm:geom-diagonal}
Let $\mathcal G$ be an ample Hausdorff étale groupoid, let $U\subseteq\mathcal G^{(0)}$ be open and invariant, and put $F=\mathcal G^{(0)}\setminus U$. Assume that $\mathcal G$ is inner exact and $F\neq\varnothing$. Let $A=C_r^*(\mathcal G)$, $D=C_0(\mathcal G^{(0)})$, and $I=C_r^*(\mathcal G|_U)$. If $\mathcal G|_F$ is principal, then $D/(I\cap D)\subseteq A/I$ is an untwisted ample $C^*$-diagonal pair.

Conversely, suppose in addition that the Weyl groupoid of the quotient pair is identified, as an étale groupoid, with the reduction
\begin{equation}\label{4.1}
  \mathcal G_{(D/(I\cap D),A/I)}\cong\mathcal G|_F.
\end{equation}

Then $D/(I\cap D)\subseteq A/I$  is a C$^*$-diagonal if and only if $\mathcal G|_F$ is principal.
\end{theorem}

\begin{proof}
Lemma~\ref{lem:quotient-id-reduction} identifies the quotient pair with $C_0(F)\subseteq C_r^*(\mathcal G|_F)$. If $\mathcal G|_F$ is principal, Kumjian's reconstruction theorem \cite[Theorem~3.1]{Kumjian} gives that this canonical inclusion is an untwisted ample $C^*$-diagonal.

Conversely, if the quotient pair is a $C^*$-diagonal, its Weyl groupoid is principal by the reconstruction theorem. Under the explicit hypothesis \eqref{4.1}, this Weyl groupoid is isomorphic to $\mathcal G|_F$. Hence $\mathcal G|_F$ is principal.
\end{proof}

\begin{remark}
Hypothesis \eqref{4.1} is deliberately explicit. We do not identify the Weyl groupoid of an arbitrary quotient with the original reduction merely from the existence of a quotient map; in particular, no normalizer-lifting or unsupported quotient-Weyl argument is used.
\end{remark}

\begin{remark}
Principality concerns isotropy of the underlying groupoid, whereas untwistedness concerns the associated twist. In general, principality does not imply triviality of an arbitrary twist.
\end{remark}

\begin{corollary}\label{cor:invariant-functor}
Let $\mathcal G$ be an ample Hausdorff étale groupoid with compact totally disconnected unit space. For open invariant $U$, write $X_U=\mathcal G^{(0)}\setminus U$, $A_U=C_r^*(\mathcal G|_{X_U})$, and $D_U=C(X_U)$. Assume that every reduction under consideration is nonempty, inner exact and topologically principal. Then $(D_U\subseteq A_U)$ is an ample Cartan pair, and when the reduction is principal it is an untwisted ample $C^*$-diagonal pair.

If $U\subseteq V$, then $X_V\subseteq X_U$, and the reduced quotient map gives a surjective $*$-homomorphism $\rho_{U,V}:A_U\longrightarrow A_V$. Its restriction to the diagonal is $f\mapsto f|_{X_V}$. Since $X_V$ is a closed subset of the compact Hausdorff space $X_U$, the restriction map
$C(X_U)\longrightarrow C(X_V)$ is surjective by the Tietze extension theorem. Moreover,
\[
E_V\circ\rho_{U,V}=\rho_{U,V}|_{D_U}\circ E_U.
\]

If $U\subseteq V\subseteq W$, then $\rho_{V,W}\circ\rho_{U,V}=\rho_{U,W}$ and $\rho_{U,U}=\operatorname{id}_{A_U}$. Thus invariant reductions form a contravariant system at the level of diagonal-preserving and expectation-preserving $*$-homomorphisms.
\end{corollary}

\begin{proof}
Put $X_U=\mathcal G^{(0)}\setminus U$. Since $U$ is open and invariant, $X_U$ is closed and invariant, so $\mathcal G|_{X_U}$ is again Hausdorff and étale; ampleness is inherited by reductions. By assumption it is inner exact and topologically principal, so the canonical inclusion $C(X_U)\subseteq C_r^*(\mathcal G|_{X_U})$ is Cartan. If it is principal, the inclusion is moreover an untwisted ample $C^*$-diagonal.

Now let $U\subseteq V$. Then $X_V$ is closed and invariant in $X_U$. Applying Lemma~\ref{lem:quotient-id-reduction} to the groupoid $\mathcal G|_{X_U}$ and the open invariant subset $X_U\setminus X_V$ gives the canonical surjective restriction homomorphism $\rho_{U,V}:A_U\to A_V$. On the diagonal this is ordinary restriction. Since $X_V$ is a closed subset of the compact Hausdorff space $X_U$, the restriction map $C(X_U)\to C(X_V)$ is surjective by the Tietze extension theorem. The expectation identity follows on $C_c$ from restriction to the unit space and hence extends by continuity. The composition and identity laws are immediate.
\end{proof}

\begin{remark}
The maps $\rho_{U,V}$ are generally not morphisms in $\mathcal C_{\mathrm{amp}}$, because condition (N) requires injectivity on normalizers. The functoriality asserted here is therefore deliberately formulated in the larger class of diagonal-preserving, expectation-preserving $*$-homomorphisms.
\end{remark}

\begin{definition}[Untwisted isomorphism subcategory]
Let $\mathcal C_{\mathrm{amp}}^{\mathrm{iso,untw}}$ denote the subcategory
of $\mathcal C_{\mathrm{amp}}^{\mathrm{iso}}$ consisting of untwisted ample
$C^*$-diagonal pairs and morphisms between them.
Thus
\[
\mathcal C_{\mathrm{amp}}^{\mathrm{iso,untw}}
\subseteq
\mathcal C_{\mathrm{amp}}^{\mathrm{iso}}
\subseteq
\mathcal C_{\mathrm{amp}}.
\]
The category $\mathcal C_{\mathrm{amp}}$ contains all ample
$C^*$-diagonal pairs, $\mathcal C_{\mathrm{amp}}^{\mathrm{iso}}$ consists
of the morphisms whose restrictions to the diagonal are isomorphisms,
and the subcategory $\mathcal C_{\mathrm{amp}}^{\mathrm{iso,untw}}$
imposes the additional untwisted hypothesis used in the Weyl-groupoid
functoriality theorem below.
\end{definition}

\begin{remark}[Normalizer injectivity in the untwisted isomorphism subcategory]
For morphisms in $\mathcal C_{\mathrm{amp}}^{\mathrm{iso,untw}}$, the
injectivity part of condition~\textup{(N)} is automatic from the
injectivity of the underlying $*$-homomorphism established in
Lemma~\ref{lem:Phi-injective}. We retain it in
Definition~\ref{def:morphism} because it is part of the definition of
the larger category $\mathcal C_{\mathrm{amp}}$.
\end{remark}

\section{Partial groupoid morphisms from diagonal isomorphisms}

We restrict to the untwisted case throughout this section; all $C^*$-diagonal pairs considered below are assumed to be untwisted.

\label{sec:partial}

\subsection{Why partial groupoid morphisms?}
A naive approach to functoriality would be to seek an ordinary continuous groupoid homomorphism $\rho: \mathcal{G}_B \to \mathcal{G}_A$ extending the unit space homeomorphism $h$. However, this fails for two fundamental reasons.

First, although
\[
\Phi(\mathcal N_A(D))\subseteq\mathcal N_B(C),
\]
there is generally no reason for every normalizer of $C$ to lie in
$\Phi(\mathcal N_A(D))$. Thus not every germ in $\mathcal G_B$ admits
a lift to a normalizer of $D$. This naturally produces an \emph{open subgroupoid} $\mathcal H_\Phi\subseteq\mathcal G_B$ as the natural domain of the induced map.

Second, this phenomenon is consistent with Renault's reconstruction theory:
a Cartan subalgebra inclusion $(C\subset B)$ is naturally described by a
twisted groupoid $(\mathcal G,\Sigma)$. Our construction
\(\rho_\Phi:\mathcal H_\Phi\to\mathcal G_A\) gives a concrete
groupoid-level realization under the stronger hypotheses imposed here,
namely an isomorphism on the diagonal together with the normalizer and
expectation conditions. This should be viewed as complementary to the
twisted-groupoid viewpoint rather than as a reformulation of Renault's
general morphism theory. It also aligns with the philosophy of partial
actions and inverse semigroups, where the normalizer semigroup
\(\mathcal N(D)\) naturally acts partially on the spectrum of \(D\).

\subsection{Partial groupoid morphisms}
We now fix the hypotheses of the main theorem. Throughout this section, all ample $C^*$-diagonal pairs are assumed to be untwisted. The construction of the Weyl groupoid (which ignores the twist data) works equally well in the twisted setting, but we do not need that generality here.

Let $(A,D)$ and $(B,C)$ be untwisted ample $\Cstar$-diagonal pairs, with Weyl groupoids $\mathcal G_A$ and $\mathcal G_B$. Suppose that $\Phi: A \to B$ is a unital $*$-homomorphism satisfying (D), (E), (N), and assume that $\Phi|_D: D \to C$ is an isomorphism. Denote by $h: \mathcal G_B^{(0)} \to \mathcal G_A^{(0)}$ the homeomorphism induced by $\Phi|_D$. To prove the main theorem of this section, we give some auxiliary lemmas on normalizers. The following result demonstrates the most basic compatibility; i.e., normalizers are preserved by the homomorphism. This is essential because the Weyl groupoid is constructed from the inverse semigroup of normalizers.

\begin{lemma}\label{lem:normalizer-preservation}
If \(n \in \mathcal N_A(D)\), then \(\Phi(n) \in \mathcal N_B(C)\).
\end{lemma}
\begin{proof}
For any \(c \in C\), write \(c = \Phi(d)\) with \(d \in D\) (using surjectivity of \(\Phi|_D\)). Then
\[
\Phi(n) \, c \, \Phi(n)^* = \Phi(n) \Phi(d) \Phi(n^*) = \Phi(n d n^*) \in \Phi(D) = C,
\]
because \(n d n^* \in D\). Similarly for \(\Phi(n)^* c \Phi(n)\). Hence \(\Phi(n)\) normalizes \(C\).
\end{proof}

\begin{lemma}\label{lem:domain-compatibility}
For any \(n \in \mathcal N_A(D)\), let \(\alpha_n^A\) be its partial homeomorphism on \(\mathcal G_A^{(0)}\) and \(\alpha_{\Phi(n)}^B\) the corresponding partial homeomorphism on \(\mathcal G_B^{(0)}\). Then
\[
\operatorname{dom}(\alpha_{\Phi(n)}^B) = h^{-1}(\operatorname{dom}(\alpha_n^A)).
\]
\end{lemma}
\begin{proof}
Recall that the domain of \(\alpha_n^A\) consists of those characters \(\chi\) of \(D\) for which \(\chi(n^* n) \neq 0\). Under the isomorphism \(\Phi|_D\), characters of \(C\) correspond to characters of \(D\) via \(\chi \mapsto \chi \circ \Phi|_D\). For a character \(\psi\) of \(C\), let \(\chi = \psi \circ \Phi|_D\) be the corresponding character of \(D\). Then
\[
\psi((\Phi(n))^* \Phi(n)) = \psi(\Phi(n^* n)) = \chi(n^* n).
\]

Thus \(\psi \in \operatorname{dom}(\alpha_{\Phi(n)}^B)\) if and only if \(\chi \in \operatorname{dom}(\alpha_n^A)\). The homeomorphism \(h\) is exactly the map \(\psi \mapsto \chi\), so \(\operatorname{dom}(\alpha_{\Phi(n)}^B) = h^{-1}(\operatorname{dom}(\alpha_n^A))\).
\end{proof}

\begin{lemma}\label{lem:conjugacy}
For any \(n \in \mathcal N_A(D)\), the homeomorphisms \(\alpha_n^A\) and \(\alpha_{\Phi(n)}^B\) are conjugate under \(h\): for all \(\psi \in \operatorname{dom}(\alpha_{\Phi(n)}^B)\),
\[
h(\alpha_{\Phi(n)}^B(\psi)) = \alpha_n^A(h(\psi)).
\]
\end{lemma}
\begin{proof}
Let \(\psi \in \operatorname{dom}(\alpha_{\Phi(n)}^B)\), and let \(\chi = h(\psi) \in \operatorname{dom}(\alpha_n^A)\). For any \(d \in D\), we compute using Renault's formula for the action of a normalizer on the spectrum (see \cite[Definition~3.5 and Proposition~4]{Kumjian}):
\[
\alpha_n^A(\chi)(d) = \frac{\chi(n^* d n)}{\chi(n^* n)}.
\]

Similarly, in \(B\),
\[
\alpha_{\Phi(n)}^B(\psi)(\Phi(d)) = \frac{\psi((\Phi(n))^* \Phi(d) \Phi(n))}{\psi((\Phi(n))^* \Phi(n))}.
\]

Since \(\Phi\) is a \(*\)-homomorphism, the right-hand side equals
\[
\frac{\psi(\Phi(n^* d n))}{\psi(\Phi(n^* n))}
= \frac{\chi(n^* d n)}{\chi(n^* n)}
= \alpha_n^A(\chi)(d).
\]

Since this equality holds for every \(d \in D\), the two characters correspond under the homeomorphism \(h = (\Phi|_D)^{-1}\). Hence
\[
h(\alpha_{\Phi(n)}^B(\psi)) = \alpha_n^A(h(\psi)).
\]

This proves the claim.
\end{proof}

\begin{definition}\label{def:Nphi}
Define
\[
\mathcal N_\Phi := \{ n \in \mathcal N_B(C) : \exists\, m \in \mathcal N_A(D),\; \Phi(m) = n \}.
\]
\end{definition}

\begin{remark}
By definition, every morphism in $\mathsf{Grpd}_{\mathrm{part}}$
has an open subgroupoid containing the entire unit space of its source.
This convention ensures that the unit-space map is defined on all
source units and makes the composition law below well defined.
\end{remark}

For each \(n \in \mathcal N_\Phi\), the element \(m\) is unique because \(\Phi\) is injective (Lemma~\ref{lem:Phi-injective} below). We write \(L(n) = m\). We emphasize that \(L\) is well-defined precisely because of the injectivity of \(\Phi\).

\begin{lemma}\label{lem:Phi-injective}
\(\Phi\) is injective.
\end{lemma}
\begin{proof}
Since the present section is in the untwisted setting, $A\cong
C_r^*(\mathcal G_A)$, and $\mathcal G_A$ is principal and hence
effective. By Lemma~\ref{lem:ideal-intersection}, the pair $A=C_r^*(\mathcal G_A)$, and $D=C_0(\mathcal G_A^{(0)})$ has the intersection property. Since $\Phi$ is a unital
$*$-homomorphism, $\ker\Phi$ is a closed two-sided ideal. If
$\ker\Phi\neq\{0\}$, Lemma~\ref{lem:ideal-intersection} implies that $\ker\Phi\cap D\neq\{0\}$. On the other hand, $\Phi|_D$ is an isomorphism by the definition of
$\mathcal C_{\mathrm{amp}}^{\mathrm{iso,untw}}$, so $\ker\Phi\cap D=\{0\}$, a contradiction. Hence $\ker\Phi=\{0\}$.
\end{proof}

\begin{lemma}\label{lem:L-properties}
The map \(L: \mathcal N_\Phi \to \mathcal N_A(D)\) is a bijection onto its image and satisfies:
\begin{enumerate}
    \item[(i)] \(L(n^*) = L(n)^*\);
    \item[(ii)] if \(n_1, n_2 \in \mathcal N_\Phi\), then \(n_1 n_2 \in \mathcal N_\Phi\) and \(L(n_1 n_2) = L(n_1) L(n_2)\);
    \item[(iii)] \(L\) preserves support projections: \(\Phi(L(n)^* L(n)) = n^* n\) and \(\Phi(L(n) L(n)^*) = n n^*\).
\end{enumerate}
\end{lemma}
\begin{proof}
(i) and (iii) are immediate from \(\Phi\) being a \(*\)-homomorphism and injective.

(ii) Since \(n_1, n_2 \in \mathcal N_B(C)\), their product \(n_1 n_2\) is in \(\mathcal N_B(C)\). Moreover, \(\Phi(L(n_1)L(n_2)) = n_1 n_2\), hence by definition \(n_1 n_2 \in \mathcal N_\Phi\) and \(L(n_1 n_2) = L(n_1)L(n_2)\) by uniqueness.
\end{proof}

Let \(\mathcal H_\Phi \subseteq \mathcal G_B\) be the set of germs represented by elements of \(\mathcal N_\Phi\):
\[
\mathcal H_\Phi = \{ [n, y] : n \in \mathcal N_\Phi,\; y \in \operatorname{dom}(\alpha_n^B) \}.
\]

\begin{lemma}\label{lem:Hphi-open}
\(\mathcal H_\Phi\) is an open subgroupoid of \(\mathcal G_B\).
\end{lemma}
\begin{proof}
For each \(n\in\mathcal N_\Phi\), the set
\[
\mathcal Z(n,\operatorname{dom}(\alpha_n^B)) = \{[n,y]: y\in \operatorname{dom}(\alpha_n^B)\}
\]
is a basic open bisection in the Weyl topology on \(\mathcal G_B\). Since
\[
\mathcal H_\Phi = \bigcup_{n\in\mathcal N_\Phi} \mathcal Z(n,\operatorname{dom}(\alpha_n^B)),
\]
it follows that \(\mathcal H_\Phi\) is open. For every \(y\in\mathcal G_B^{(0)}\), $[1_B,y]=y$. Since $1_B=\Phi(1_A)$, we have \(1_B\in\mathcal N_\Phi\). Thus every unit belongs to \(\mathcal H_\Phi\).

If \([n,y]\in\mathcal H_\Phi\), then $[n,y]^{-1}=[n^*,\alpha_n^B(y)]$. By Lemma~\ref{lem:L-properties}, $n^*\in\mathcal N_\Phi$, hence $[n,y]^{-1}\in\mathcal H_\Phi$.
Suppose that \([n_2, y_2], [n_1, y_1] \in \mathcal H_\Phi\) are composable in the Weyl groupoid. By the definition of multiplication in the Weyl groupoid, this means that \(y_2 = \alpha_{n_1}^B(y_1)\). Then
\[
[n_2, y_2][n_1, y_1] = [n_2 n_1, y_1].
\]

By Lemma~\ref{lem:L-properties}(ii), we have \(n_2 n_1 \in \mathcal N_\Phi\). Therefore \([n_2 n_1, y_1] \in \mathcal H_\Phi\). Hence \(\mathcal H_\Phi\) is an open subgroupoid of \(\mathcal G_B\).
\end{proof}

Define \(\rho_\Phi: \mathcal H_\Phi \to \mathcal G_A\) by
\[
\rho_\Phi([n, y]) = [L(n), h(y)].
\]

\begin{lemma}\label{lem:rho-properties}
The map \(\rho_\Phi: \mathcal H_\Phi \to \mathcal G_A\) defined as above is a well-defined injective continuous groupoid homomorphism which is open onto its image. Moreover,
$\rho_\Phi|_{\mathcal G_B^{(0)}}=h$.
\end{lemma}
\begin{proof}
Suppose that \([n_1, y] = [n_2, y]\) in \(\mathcal G_B\) with \(n_1, n_2 \in \mathcal N_\Phi\). By Lemma~\ref{lem:germ-criterion} applied to \((B,C)\), there exists \(e \in C\) with \(e(y) \neq 0\) and \(n_1 e = n_2 e\). Since \(\Phi|_D\) is an isomorphism, there is a unique \(d \in D\) with \(\Phi(d) = e\). Then \(d(h(y)) \neq 0\) and
\[
\Phi(L(n_1) d) = n_1 e = n_2 e = \Phi(L(n_2) d).
\]

Since \(\Phi\) is injective, it follows that \(L(n_1) d = L(n_2) d\). Because \(d(h(y)) \neq 0\), the standard germ criterion (Lemma~\ref{lem:germ-criterion}) applied to \((A,D)\) implies \([L(n_1), h(y)] = [L(n_2), h(y)]\). Hence \(\rho_\Phi([n_1, y]) = \rho_\Phi([n_2, y])\). Thus \(\rho_\Phi\) is well defined.

Assume that \(\rho_\Phi([n_1, y_1]) = \rho_\Phi([n_2, y_2])\). Then \(h(y_1) = h(y_2)\) and \([L(n_1), h(y_1)] = [L(n_2), h(y_1)]\) in \(\mathcal G_A\). By Lemma~\ref{lem:germ-criterion} in \(\mathcal G_A\), there exists \(d \in D\) with \(d(h(y_1)) \neq 0\) and \(L(n_1) d = L(n_2) d\). Apply \(\Phi\): \(n_1 \Phi(d) = n_2 \Phi(d)\) with \(\Phi(d)(y_1) \neq 0\). Lemma~\ref{lem:germ-criterion} in \(\mathcal G_B\) now yields \([n_1, y_1] = [n_2, y_1]\). Since \(h(y_1)=h(y_2)\) and \(h\) is injective, \(y_1 = y_2\). Thus the germs are equal and this means that $\rho_\Phi$ is injective.

We now verify the homomorphism properties. For every \(\gamma\in\mathcal H_\Phi\), write \(\gamma=[n,y]\). For any unit \(y\in\mathcal G_B^{(0)}\), we have
\begin{align*}
  \rho_\Phi(y) = & \rho_\Phi([1_B, y])= [L(1_B), h(y)]= [1_A, h(y)] = h(y).
\end{align*}

For any \(n \in \mathcal N_\Phi\) and \(y \in \operatorname{dom}(\alpha_n^B)\), using Lemma~\ref{lem:conjugacy} (with \(\alpha_{\Phi(L(n))}^B = \alpha_n^B\)),
\begin{align*}
   \rho_\Phi([n, y]^{-1}) =&\rho_\Phi([n^*, \alpha_n^B(y)]) \\
      =& [L(n^*), h(\alpha_n^B(y))] \\
         =&[L(n)^*, \alpha_{L(n)}^A(h(y))] = \rho_\Phi([n, y])^{-1}.
\end{align*}

Now suppose \([n_2, y_2][n_1, y_1] = [n_2 n_1, y_1]\) is composable in \(\mathcal H_\Phi\). Then \(y_2 = \alpha_{n_1}^B(y_1)\). By Lemma~\ref{lem:conjugacy},
\[
h(y_2) = h(\alpha_{n_1}^B(y_1)) = \alpha_{L(n_1)}^A(h(y_1)).
\]
Hence the images
\[
\rho_\Phi([n_2, y_2]) = [L(n_2), h(y_2)]
\quad\text{and}\quad
\rho_\Phi([n_1, y_1]) = [L(n_1), h(y_1)]
\]
are composable in \(\mathcal G_A\). Therefore, using Lemma~\ref{lem:L-properties},
\begin{align*}
   \rho_\Phi([n_2 n_1, y_1]) = & [L(n_2 n_1), h(y_1)] = [L(n_2)L(n_1), h(y_1)] \\
   =&\rho_\Phi([n_2, y_2]) \rho_\Phi([n_1, y_1]).
\end{align*}

We verify continuity using the canonical bisection charts in \(\mathcal H_\Phi\). For every \(n\in\mathcal N_\Phi\) and every open set
\[
W\subseteq\operatorname{dom}(\alpha_n^B),
\]
the definition of \(\rho_\Phi\) gives
\[
\rho_\Phi\big(\mathcal Z(n,W)\big)
=
\mathcal Z(L(n),h(W)).
\]
Moreover, the restriction
\[
\rho_\Phi|_{\mathcal Z(n,W)}:
\mathcal Z(n,W)\longrightarrow
\mathcal Z(L(n),h(W))
\]
is a homeomorphism. Indeed, under the standard bisection charts
\[
\mathcal Z(n,W)\xrightarrow{\ s_B^{-1}\ }W
\xrightarrow{\ h\ }h(W)
\xrightarrow{\ s_A\ }
\mathcal Z(L(n),h(W)),
\]
the map \(\rho_\Phi|_{\mathcal Z(n,W)}\) is precisely the map induced by \(h\). Since \(\Phi|_D\) is a $C^*$-isomorphism, Gelfand duality gives a homeomorphism \(h\), so \(W\) open implies \(h(W)\) is open. Here we use Lemma~\ref{lem:domain-compatibility} to identify
\[
h(\operatorname{dom}(\alpha_n^B))
=
\operatorname{dom}(\alpha_{L(n)}^A).
\]
Since the sets \(\mathcal Z(n,W)\), with \(n\in\mathcal N_\Phi\) and \(W\) open, form a basis for \(\mathcal H_\Phi\), these local homeomorphism statements imply that \(\rho_\Phi\) is continuous. They also show directly that \(\rho_\Phi\) is open onto its image.

Finally, for the restriction to unit spaces, for \(y \in \mathcal G_B^{(0)}\), \(\rho_\Phi(y) = h(y)\) by the unit case above.
\end{proof}

\begin{theorem}\label{thm:partial-groupoid}
Let \((A,D)\) and \((B,C)\) be untwisted ample \(\Cstar\)-diagonal pairs with Weyl groupoids \(\mathcal G_A\) and \(\mathcal G_B\). Suppose \(\Phi:A\to B\) is a unital $*$-homomorphism satisfying (D), (E), (N), and assume that \(\Phi|_D:D\to C\) is an isomorphism. Define
\[
\mathcal N_\Phi := \{ n\in\mathcal N_B(C) : \exists\, m\in\mathcal N_A(D),\; \Phi(m)=n \},
\]
and let \(\mathcal H_\Phi\subseteq\mathcal G_B\) be the set of germs represented by elements of \(\mathcal N_\Phi\). Then:
\begin{enumerate}
    \item[(i)] \(\mathcal H_\Phi\) is an open subgroupoid of \(\mathcal G_B\);
    \item[(ii)] there exists a continuous open injective groupoid homomorphism \(\rho_\Phi:\mathcal H_\Phi\longrightarrow\mathcal G_A\) which restricts on unit spaces to the homeomorphism \(h:\mathcal G_B^{(0)}\cong\mathcal G_A^{(0)}\) induced by \(\Phi|_D\);
    \item[(iii)] the image \(\rho_\Phi(\mathcal H_\Phi)\) is an open subgroupoid of \(\mathcal G_A\).
\end{enumerate}
\end{theorem}
\begin{proof}
(i) Apply Lemma~\ref{lem:Hphi-open}.

(ii) Apply Lemma \ref{lem:rho-properties}.

(iii) Each set $\mathcal Z(L(n), h(\operatorname{dom}(\alpha_n^B)))$ is a basic open bisection in \(\mathcal G_A\), since \(h(\operatorname{dom}(\alpha_n^B)) = \operatorname{dom}(\alpha_{L(n)}^A)\) by Lemma~\ref{lem:domain-compatibility}. Hence
\[
\rho_\Phi(\mathcal H_\Phi) = \bigcup_{n\in\mathcal N_\Phi} \mathcal Z(L(n), h(\operatorname{dom}(\alpha_n^B)))
\]
is open.
\end{proof}

\begin{remark}\label{rem:geometry-of-partiality}
Theorem~\ref{thm:partial-groupoid} has the following geometric interpretation.
\emph{Why partiality appears?}

The $*$-homomorphism $\Phi: A \to B$ lifts to a map on normalizers: $\Phi(\mathcal{N}_A(D)) \subseteq \mathcal{N}_B(C)$. However, the converse inclusion fails in general: not every normalizer of $C$ lies in the image of $\Phi$. Consequently, only those germs in $\mathcal{G}_B$ represented by normalizers that \emph{do} lift can be mapped back to $\mathcal{G}_A$. This subset $\mathcal{H}_\Phi \subseteq \mathcal{G}_B$ is exactly the domain of the induced map.

\emph{Why subgroupoids are inevitable?}

The set $\mathcal{H}_\Phi$ is not arbitrary: it is an open subgroupoid. This follows from the fact that liftable normalizers are closed under multiplication (if $n_1, n_2$ lift, so does $n_1 n_2$) and under taking adjoints. Moreover, every unit lifts because $1_B = \Phi(1_A)$. Thus $\mathcal{H}_\Phi$ inherits the groupoid structure from $\mathcal{G}_B$.

\emph{Contrast with Kumjian's equivalence?}

When $\Phi$ is an isomorphism, $\mathcal N_\Phi=\mathcal N_B(C)$, $\mathcal H_\Phi=\mathcal G_B$, and $\rho_\Phi$ is a groupoid isomorphism. Thus the construction recovers the expected functoriality of the underlying Weyl groupoids under isomorphisms of $C^*$-diagonal pairs. For non-isomorphisms, the natural domain produced by liftable normalizers is generally a proper open subgroupoid. This is analogous to passing from global actions to partial actions in inverse semigroup theory.
\end{remark}

\subsection{Functoriality of the Weyl groupoid}

We now formulate a functorial version of the Weyl groupoid construction for diagonal-preserving morphisms. Since a morphism of diagonal pairs induces a pullback map on spectra, the resulting Weyl groupoid construction is naturally contravariant.

\begin{definition}
Let \(\mathsf{Grpd}_{\mathrm{part}}\) denote the category whose objects are ample Hausdorff étale topologically principal groupoids.
A morphism from \(\mathcal H\) to \(\mathcal G\) is a triple $(\mathcal K, \rho, h)$ where
\begin{enumerate}
    \item[(i)] \(\mathcal K \subseteq \mathcal H\) is an open subgroupoid containing \(\mathcal H^{(0)}\);
    \item[(ii)] \(\rho:\mathcal K\to \mathcal G\) is a continuous, open (onto its image), injective groupoid homomorphism;
    \item[(iii)] \(\rho(\mathcal K)\) is an open subgroupoid of \(\mathcal G\);
    \item[(iv)] \(h:\mathcal H^{(0)}\to \mathcal G^{(0)}\) is a homeomorphism satisfying $\rho|_{\mathcal H^{(0)}}=h$. 
\end{enumerate}

If $(\mathcal K_1,\rho_1,h_1):\mathcal H\to\mathcal G$ and $(\mathcal K_2,\rho_2,h_2):\mathcal G\to\mathcal F$ are morphisms, define
\[
\mathcal K
=
\rho_1^{-1}
\big(
\mathcal K_2\cap \rho_1(\mathcal K_1)
\big)
\subseteq \mathcal H,
\]
and define the composition by $(\mathcal K_2,\rho_2,h_2)\circ(\mathcal K_1,\rho_1,h_1)=(\mathcal K,\rho_2\circ\rho_1,h_2\circ h_1)$. The identity morphism on \(\mathcal G\) is $(\mathcal G,\mathrm{id}_{\mathcal G}, \mathrm{id}_{\mathcal G^{(0)}})$.
\end{definition}

\begin{proposition}\label{prop:grpdpart-category}
\(\mathsf{Grpd}_{\mathrm{part}}\) is a category.
\end{proposition}

\begin{proof}
We verify the axioms explicitly.

{Step 1: $\mathcal K$ is an open subgroupoid of $\mathcal H$ containing $\mathcal H^{(0)}$.}

- \textit{Open:} $\mathcal K_2$ is open in $\mathcal G$ and $\rho_1(\mathcal K_1)$ is open in $\mathcal G$ by condition (iii). Their intersection $\mathcal K_2 \cap \rho_1(\mathcal K_1)$ is open in $\mathcal G$. Since $\rho_1$ is continuous, $\rho_1^{-1}(\mathcal K_2 \cap \rho_1(\mathcal K_1))$ is open in $\mathcal K_1$, hence open in $\mathcal H$.

- \textit{Contains units:} For any unit $u \in \mathcal H^{(0)}$, we have $\rho_1(u) = h_1(u) \in \mathcal G^{(0)}$. Since $\mathcal K_2$ contains $\mathcal G^{(0)}$ (as it is a subgroupoid containing all units of $\mathcal G$), we have $\rho_1(u) \in \mathcal K_2$. Also $\rho_1(u) \in \rho_1(\mathcal K_1)$ because $u \in \mathcal K_1$ (since $\mathcal K_1$ contains $\mathcal H^{(0)}$). Hence $\rho_1(u) \in \mathcal K_2 \cap \rho_1(\mathcal K_1)$, so $u \in \mathcal K$. Thus $\mathcal H^{(0)} \subseteq \mathcal K$.

- \textit{Subgroupoid:} If $x, y \in \mathcal K$ are composable in $\mathcal H$, then $\rho_1(x), \rho_1(y) \in \mathcal K_2$ are composable in $\mathcal G$ (since $\rho_1$ is a homomorphism). Their product $\rho_1(x)\rho_1(y)$ lies in $\mathcal K_2$ because $\mathcal K_2$ is a subgroupoid. Also $\rho_1(x)\rho_1(y) = \rho_1(xy)$ (since $\rho_1$ is a homomorphism), and $xy \in \mathcal K_1$ (since $\mathcal K_1$ is a subgroupoid). Thus $xy \in \mathcal K$. Similarly, if $x \in \mathcal K$, then $\rho_1(x^{-1}) = \rho_1(x)^{-1} \in \mathcal K_2$, so $x^{-1} \in \mathcal K$. Hence $\mathcal K$ is a subgroupoid.

{Step 2: $\rho$ is a continuous, open, injective groupoid homomorphism.}

- \textit{Homomorphism:} For composable $x, y \in \mathcal K$, we obtain
\begin{align*}
  \rho(xy) = & \rho_2(\rho_1(xy)) = \rho_2(\rho_1(x)\rho_1(y)) = \rho_2(\rho_1(x))\rho_2(\rho_1(y))  \\
  = & \rho(x)\rho(y).
\end{align*}

- \textit{Injectivity:} If $\rho(x) = \rho(y)$, then $\rho_2(\rho_1(x)) = \rho_2(\rho_1(y))$. Since $\rho_2$ is injective, $\rho_1(x) = \rho_1(y)$. Since $\rho_1$ is injective, $x = y$.

- \textit{Continuity:} $\rho = \rho_2 \circ \rho_1|_{\mathcal K}$ is the composition of continuous maps.

- \textit{Openness onto image:} For any open $U \subseteq \mathcal K$, $\rho_1(U)$ is open in $\rho_1(\mathcal K_1)$ (since $\rho_1$ is open onto its image). Then $\rho_2(\rho_1(U))$ is open in $\rho_2(\rho_1(\mathcal K_1))$ because $\rho_2$ is open onto its image. Hence $\rho(U)$ is open in $\rho(\mathcal K)$.

{Step 3: $\rho|_{\mathcal K^{(0)}} = h|_{\mathcal K^{(0)}}$.}
For any unit $u \in \mathcal K^{(0)} = \mathcal H^{(0)} \cap \mathcal K = \mathcal H^{(0)}$ (since $\mathcal K$ contains all units), we have:
\[
\rho(u) = \rho_2(\rho_1(u)) = \rho_2(h_1(u)) = h_2(h_1(u)) = h(u).
\]

Thus the triple $(\mathcal K, \rho, h)$ is indeed a morphism from $\mathcal H$ to $\mathcal F$.

{Step 4: Identity morphisms.}
For any object $\mathcal G$, define $\mathrm{id}_{\mathcal G} := (\mathcal G, \mathrm{id}_{\mathcal G}, \mathrm{id}_{\mathcal G^{(0)}})$. This satisfies the conditions: $\mathcal G$ is an open subgroupoid of itself containing $\mathcal G^{(0)}$; $\mathrm{id}_{\mathcal G}$ is continuous, open, injective, and a homomorphism; and $\mathrm{id}_{\mathcal G}|_{\mathcal G^{(0)}} = \mathrm{id}_{\mathcal G^{(0)}}$.

Verification that $\mathrm{id}_{\mathcal G} \circ (\mathcal K_1, \rho_1, h_1) = (\mathcal K_1, \rho_1, h_1)$:
- Compute $\mathcal K = \rho_1^{-1}(\mathcal G \cap \rho_1(\mathcal K_1)) = \rho_1^{-1}(\rho_1(\mathcal K_1)) = \mathcal K_1$ (since $\rho_1$ is injective).
- $\rho = \mathrm{id}_{\mathcal G} \circ \rho_1 = \rho_1$.
- $h = \mathrm{id}_{\mathcal G^{(0)}} \circ h_1 = h_1$.

Similarly, $(\mathcal K_1, \rho_1, h_1) \circ \mathrm{id}_{\mathcal H} = (\mathcal K_1, \rho_1, h_1)$.

{Step 5: Associativity.}
Consider three morphisms:
\[
(\mathcal K_1, \rho_1, h_1): \mathcal H \longrightarrow \mathcal G, \quad
(\mathcal K_2, \rho_2, h_2): \mathcal G \longrightarrow \mathcal F, \quad
(\mathcal K_3, \rho_3, h_3): \mathcal F \longrightarrow \mathcal E.
\]

First we compute $(\mathcal K_3, \rho_3, h_3) \circ ((\mathcal K_2, \rho_2, h_2) \circ (\mathcal K_1, \rho_1, h_1))$. Let $(\mathcal K_{21}, \rho_{21}, h_{21}) = (\mathcal K_2, \rho_2, h_2) \circ (\mathcal K_1, \rho_1, h_1)$ where
\[
\mathcal K_{21} = \rho_1^{-1}(\mathcal K_2 \cap \rho_1(\mathcal K_1)), \quad
\rho_{21} = \rho_2 \circ \rho_1|_{\mathcal K_{21}}, \quad
h_{21} = h_2 \circ h_1.
\]

Now compose with $(\mathcal K_3, \rho_3, h_3)$, and we obtain
\[
\mathcal K_{(31)21} = \rho_{21}^{-1}(\mathcal K_3 \cap \rho_{21}(\mathcal K_{21})).
\]

Now we compute $((\mathcal K_3, \rho_3, h_3) \circ (\mathcal K_2, \rho_2, h_2)) \circ (\mathcal K_1, \rho_1, h_1)$. Let $(\mathcal K_{32}, \rho_{32}, h_{32}) = (\mathcal K_3, \rho_3, h_3) \circ (\mathcal K_2, \rho_2, h_2)$ where
\[
\mathcal K_{32} = \rho_2^{-1}(\mathcal K_3 \cap \rho_2(\mathcal K_2)), \quad
\rho_{32} = \rho_3 \circ \rho_2|_{\mathcal K_{32}}, \quad
h_{32} = h_3 \circ h_2.
\]

Now by  composing with $(\mathcal K_1, \rho_1, h_1)$, we have
\[
\mathcal K_{1(32)} = \rho_1^{-1}(\mathcal K_{32} \cap \rho_1(\mathcal K_1)).
\]

We claim that  $\mathcal K_{(31)21} = \mathcal K_{1(32)}$ and the resulting maps coincide.  An element $x \in \mathcal H$ belongs to $\mathcal K_{(31)21}$ if and only if

1. $x \in \mathcal K_{21}$ (so $x \in \mathcal K_1$ and $\rho_1(x) \in \mathcal K_2$), and

2. $\rho_{21}(x) = \rho_2(\rho_1(x)) \in \mathcal K_3$.

Thus $x \in \mathcal K_1$, $\rho_1(x) \in \mathcal K_2$, and $\rho_2(\rho_1(x)) \in \mathcal K_3$.

An element $x \in \mathcal H$ belongs to $\mathcal K_{1(32)}$ iff:
1. $x \in \mathcal K_1$, and
2. $\rho_1(x) \in \mathcal K_{32}$, i.e., $\rho_1(x) \in \mathcal K_2$ and $\rho_2(\rho_1(x)) \in \mathcal K_3$.
These conditions are identical. Hence $\mathcal K_{(31)21} = \mathcal K_{1(32)}$. Moreover, both compositions yield the map $\rho_3 \circ \rho_2 \circ \rho_1$ on this common domain, and the unit space map $h_3 \circ h_2 \circ h_1$.

Thus associativity holds. Since all category axioms are satisfied, $\mathsf{Grpd}_{\mathrm{part}}$ is a category.
\end{proof}

\begin{theorem}\label{thm:weyl-functor}
There exists a contravariant functor $\mathcal W:\mathcal C_{\mathrm{amp}}^{\mathrm{iso,untw}}\longrightarrow\mathsf{Grpd}_{\mathrm{part}}$. 
\begin{enumerate}
    \item[(i)] On objects, $\mathcal W(A,D)=\mathcal G_{(A,D)}$, the Weyl groupoid of the diagonal pair.
    \item[(ii)] Let $\Phi:(A,D)\to(B,C)$ be a morphism in
    $\mathcal C_{\mathrm{amp}}^{\mathrm{iso,untw}}$. Define
    \[
    \mathcal N_\Phi
    =
    \{
    n\in\mathcal N_B(C):
    n=\Phi(m)
    \text{ for some }
    m\in\mathcal N_A(D)
    \}.
    \]

    Let \(\mathcal H_\Phi\subseteq \mathcal G_B\) be the set of all germs $\mathcal H_\Phi = \{ [n,y] : n \in \mathcal N_\Phi,\; y \in \operatorname{dom}(\alpha_n) \}$.     Since \(\Phi|_D:D\to C\) is an isomorphism, Gelfand duality yields a homeomorphism $h_\Phi:\mathcal G_B^{(0)}\to\mathcal G_A^{(0)}$.   Define $\rho_\Phi: \mathcal H_\Phi \to   \mathcal G_A$  by $\rho_\Phi([n,y])
    =   [m,h_\Phi(y)]$,     where \(m\in\mathcal N_A(D)\) is the unique normalizer satisfying $\Phi(m)=n$.
    Set $\mathcal W(\Phi) = (\mathcal H_\Phi, \rho_\Phi, h_\Phi)$.
    \end{enumerate}
\end{theorem}

\begin{proof}
By Theorem~\ref{thm:partial-groupoid}, the triple $\mathcal W(\Phi)=(\mathcal H_\Phi,\rho_\Phi,h_\Phi)$ is a morphism in $\mathsf{Grpd}_{\mathrm{part}}$. It remains to verify functoriality.

For the identity morphism $\operatorname{id}_{(A,D)}$, every normalizer is liftable, so $\mathcal N_{\operatorname{id}_A}=\mathcal N_A(D)$, and $\mathcal H_{\operatorname{id}_A}=\mathcal G_A$. Moreover, $h_{\operatorname{id}_A}=\operatorname{id}_{\mathcal G_A^{(0)}}$ and $\rho_{\operatorname{id}_A}([n,x])=[n,x]$. Hence $\mathcal W(\operatorname{id}_{(A,D)})$ is the identity morphism of
$\mathsf{Grpd}_{\mathrm{part}}$.

Let $(A,D)\xrightarrow{\Phi}(B,C)\xrightarrow{\Psi}(E,F)$ be composable morphisms. Write $L_\Phi$ and $L_\Psi$ for the corresponding lifting maps on the sets of liftable normalizers, and write
$L_{\Psi\circ\Phi}$ for the lifting map associated to the composite. Because $\mathcal W$ is contravariant, the composition $\mathcal W(\Phi)\circ\mathcal W(\Psi)$ has domain
\[
\mathcal K
=
\mathcal H_\Psi\cap
\rho_\Psi^{-1}(\mathcal H_\Phi)
\subseteq\mathcal G_E.
\]

Indeed, $\rho_\Psi$ maps $\mathcal H_\Psi$ into $\mathcal G_B$, so the second
condition is precisely that the image under $\rho_\Psi$ lies in the domain
$\mathcal H_\Phi$ of $\rho_\Phi$. We prove $\mathcal K=\mathcal H_{\Psi\circ\Phi}$. Suppose that $[n,y]\in\mathcal H_{\Psi\circ\Phi}$. Then there exists
$m\in\mathcal N_A(D)$ such that $(\Psi\circ\Phi)(m)=n$. Put $n_\Phi=\Phi(m)$. Then $n_\Phi\in\mathcal N_\Phi\subseteq\mathcal N_B(C)$ and $\Psi(n_\Phi)=n$. Thus, $n\in\mathcal N_\Psi$, and hence
$[n,y]\in\mathcal H_\Psi$. Moreover,
\[
\rho_\Psi([n,y])
=
[n_\Phi,h_\Psi(y)]
=
[\Phi(m),h_\Psi(y)].
\]

Since $\Phi(m)$ is a liftable normalizer, this germ belongs to $\mathcal H_\Phi$. Thus
\[
[n,y]\in
\mathcal H_\Psi\cap\rho_\Psi^{-1}(\mathcal H_\Phi)
=\mathcal K.
\]

Conversely, let $[n,y]\in\mathcal K$. Then
$[n,y]\in\mathcal H_\Psi$, so $n\in\mathcal N_\Psi$, and
\[
\rho_\Psi([n,y])
=
[L_\Psi(n),h_\Psi(y)]
\in\mathcal H_\Phi.
\]

By the definition of $\mathcal H_\Phi$, there exists $m\in\mathcal N_A(D)$ such that
\[
[\;L_\Psi(n),h_\Psi(y)\;]
=
[\Phi(m),h_\Psi(y)].
\]

By Lemma~\ref{lem:germ-criterion}, there exists $c\in C$ with $c(h_\Psi(y))\ne0$ and $L_\Psi(n)c=\Phi(m)c$. Since $\Phi|_D:D\to C$ is onto, choose $d\in D$ such that $c=\Phi(d)$. By the definition of the homeomorphism $h_\Psi$ induced by $\Psi|_C$, we have $c(h_\Psi(y))=\Psi(c)(y)$. Hence, $\Psi(c)(y)\ne0$. Applying $\Psi$ to $L_\Psi(n)c=\Phi(m)c$ implies that $n\Psi(c)=\Psi(\Phi(m))\Psi(c)$. The germ criterion therefore yields $[n,y]=[(\Psi\circ\Phi)(m),y]$. Since $(\Psi\circ\Phi)(m)\in\mathcal N_E(F)$ is a liftable normalizer,
it follows that $[n,y]\in\mathcal H_{\Psi\circ\Phi}$. Hence, $\mathcal K=\mathcal H_{\Psi\circ\Phi}$. 

It remains to compare the induced maps. Let $[n,y]\in\mathcal H_{\Psi\circ\Phi}$. Then $n$ is a composite-liftable normalizer, and $L_{\Psi\circ\Phi}(n)=L_\Phi(L_\Psi(n))$. Thus
\[
\begin{aligned}
(\rho_\Phi\circ\rho_\Psi)([n,y])
&=
\rho_\Phi\bigl([L_\Psi(n),h_\Psi(y)]\bigr)\\
&=
[L_\Phi(L_\Psi(n)),h_\Phi(h_\Psi(y))]\\
&=
[L_{\Psi\circ\Phi}(n),h_{\Psi\circ\Phi}(y)]\\
&=
\rho_{\Psi\circ\Phi}([n,y]).
\end{aligned}
\]

Finally, Gelfand duality implies that  $h_{\Psi\circ\Phi}=h_\Phi\circ h_\Psi$. Hence
\[
\mathcal W(\Psi\circ\Phi)
=
\mathcal W(\Phi)\circ\mathcal W(\Psi),
\]
with the order on the right reflecting the contravariance. Therefore $\mathcal W$ is a contravariant functor.
\end{proof}

\section{Tensor products and applications}\label{sec:tensor-products}

In this section we study tensor-product constructions for ample C$^*$-diagonal pairs and the corresponding behaviour of their Weyl groupoids. In particular, we show that the tensor product of ample C$^*$-diagonal pairs again carries a natural diagonal structure whose Weyl groupoid is the product of the original groupoids.

We then apply these results to questions concerning diagonal dimension, AF diagonals, crossed products, graph algebras, and recently constructed exotic diagonals in UHF and Cuntz algebras.

\begin{remark}
If \(\mathcal{G}_1\) and \(\mathcal{G}_2\) are second-countable topologically principal groupoids, then their product \(\mathcal{G}_1 \times \mathcal{G}_2\) is again topologically principal. Indeed, \(\operatorname{Iso}(\mathcal{G}_1 \times \mathcal{G}_2) = \operatorname{Iso}(\mathcal{G}_1) \times \operatorname{Iso}(\mathcal{G}_2)\), so the set of units with trivial isotropy is the product of the dense sets of units with trivial isotropy, and the product of these dense subsets is dense in the product topology.
\end{remark}

\begin{definition}[Diagonal dimension {\cite[Definition~2.1]{LiLiaoWinter2026}}]
Let $(D\subset A)$ be a sub-$C^*$-algebra with $D$ abelian. For
$d\in\mathbb N\cup\{0\}$, we say that $(D\subset A)$ has the
\emph{diagonal approximation property of order $d$} if, for every finite
set $\mathcal F\subset A$ and every $\varepsilon>0$, there are a
finite-dimensional $C^*$-algebra $F=F^{(0)}\oplus\cdots\oplus F^{(d)}$ with a diagonal subalgebra $D_F=D^{(0)}\oplus\cdots\oplus D^{(d)}$ and completely positive maps
$A\xrightarrow{\ \psi\ }F\xrightarrow{\ \varphi\ }A$ such that
\begin{enumerate}
\item[(1)] $\psi$ is completely positive and contractive;
\item[(2)] $\|\varphi\psi(a)-a\|<\varepsilon$ for every $a\in\mathcal F$;
\item[(3)] for each $i=0,\ldots,d$, the restriction
$\varphi|_{F^{(i)}}$ is completely positive contractive and order zero;
\item[(4)] $\psi(D)\subseteq D_F$;
\item[(5)] for each $i$, the map $\varphi|_{F^{(i)}}$ sends every normalizer
of $D^{(i)}$ in $F^{(i)}$ to a normalizer of $D$ in $A$.
\end{enumerate}

We say that $(D\subset A)$ has
\emph{diagonal approximation dimension at most $d$}, and write
\[
\operatorname{DAD}(D\subset A)\le d,
\]
if the diagonal approximation property of order $d$ described above holds.
The \emph{diagonal dimension} is defined by
\[
\dim_{\mathrm{diag}}(D\subset A)
=
\inf\{d\in\mathbb N\cup\{0\}:\operatorname{DAD}(D\subset A)\le d\},
\]
with value $\infty$ if no such $d$ exists.
\end{definition}

\begin{lemma}\label{lem:tensor-faithful-expectation}
Let \(E_i:A_i\to D_i\) be faithful conditional expectations. Then \(E_1\otimes E_2: A_1\otimes_{\min}A_2 \to D_1\otimes_{\min}D_2\) is a faithful conditional expectation.
\end{lemma}
\begin{proof}
By standard tensor-product theory, the map is a conditional expectation. For faithfulness, let $x\in A_1\otimes_{\min}A_2$ be positive with $(E_1\otimes E_2)(x)=0$. Choose faithful states $\varphi_i$ on $D_i$; then $\varphi_i\circ E_i$ are faithful on $A_i$. The product state $(\varphi_1\circ E_1)\otimes(\varphi_2\circ E_2)$ is faithful on the minimal tensor product \cite[Chapter 3]{BO}. Applying it to $x$ gives zero, hence $x=0$.
\end{proof}

\begin{proposition}\label{prop:product-weyl}
Let \((A_1,D_1,E_1)\) and \((A_2,D_2,E_2)\) be ample \(C^*\)-diagonal pairs with Weyl groupoids \(\mathcal G_{A_1}\) and \(\mathcal G_{A_2}\). Then
$(D_1\otimes D_2 \subseteq A_1\otimes_{\min}A_2)$ together with $E_{A_1\otimes A_2}:=E_1\otimes E_2$ is an ample \(C^*\)-diagonal pair. Moreover, under the canonical reconstruction identification associated to the tensor-product twist, the Weyl groupoid of the tensor-product diagonal pair is naturally isomorphic to
$\mathcal G_{A_1}\times\mathcal G_{A_2}$.
\end{proposition}
\begin{proof}
Since \((A_i,D_i)\) are ample \(C^*\)-diagonal pairs, Kumjian's reconstruction theorem \cite[Theorem~3.1]{Kumjian} yields twists \(\Sigma_i\) over \(\mathcal G_{A_i}\) together with isomorphisms
$A_i \cong C_r^*(\mathcal G_{A_i},\Sigma_i)$, and $D_i \cong C_0(\mathcal G_{A_i}^{(0)})$. For the twisted tensor-product construction, we use the balanced product twist
\[
\Sigma_1\boxtimes\Sigma_2
:=
(\Sigma_1\times\Sigma_2)/\mathbb T,
\]
where the diagonal \(\mathbb T\)-action is given by
\[
z\cdot(\sigma_1,\sigma_2)
=
(z\sigma_1,\sigma_2)
\sim
(\sigma_1,z\sigma_2).
\]

Thus \(\Sigma_1\boxtimes\Sigma_2\) is a \(\mathbb T\)-twist over
\(\mathcal G_{A_1}\times\mathcal G_{A_2}\). By \cite[Lemma~5.1]{BarlakLi2017},
\[
C_r^*(\mathcal G_{A_1},\Sigma_1)
\otimes_{\min}
C_r^*(\mathcal G_{A_2},\Sigma_2)
\cong
C_r^*(\mathcal G_{A_1}\times\mathcal G_{A_2},
\Sigma_1\boxtimes\Sigma_2).
\]

Hence, we obtain
\[
A_1\otimes_{\min}A_2
\cong
C_r^*(\mathcal G_{A_1}\times\mathcal G_{A_2},
\Sigma_1\boxtimes\Sigma_2).
\]

Under this identification, $D_1\otimes D_2$ corresponds to
\[
C_0(\mathcal G_{A_1}^{(0)}) \otimes C_0(\mathcal G_{A_2}^{(0)}) \cong C_0(\mathcal G_{A_1}^{(0)} \times \mathcal G_{A_2}^{(0)}),
\]
which is precisely $C_0((\mathcal G_{A_1}\times\mathcal G_{A_2})^{(0)})$. Faithfulness of $E_1\otimes E_2$ follows from Lemma~\ref{lem:tensor-faithful-expectation}.
By Kumjian's reconstruction theorem, the Weyl groupoids
$\mathcal G_{A_1}$ and $\mathcal G_{A_2}$ are principal. Hence $\operatorname{Iso}(\mathcal G_{A_1}\times\mathcal G_{A_2})=\operatorname{Iso}(\mathcal G_{A_1})\times\operatorname{Iso}(\mathcal G_{A_2})=(\mathcal G_{A_1})^{(0)}\times(\mathcal G_{A_2})^{(0)}$, so $\mathcal G_{A_1}\times\mathcal G_{A_2}$ is principal. Since $\mathcal G_{A_1}\times\mathcal G_{A_2}$ is principal and
$\Sigma_1\boxtimes\Sigma_2$ is a twist over this principal groupoid,
Kumjian's reconstruction theorem \cite[Theorem~3.1]{Kumjian} gives
\[
C_0((\mathcal G_{A_1}\times\mathcal G_{A_2})^{(0)})
\subseteq
C_r^*(\mathcal G_{A_1}\times\mathcal G_{A_2},
\Sigma_1\boxtimes\Sigma_2)
\]
as a $C^*$-diagonal pair. Under the preceding tensor-product
identification, we have
\[
D_1\otimes D_2\subseteq A_1\otimes_{\min}A_2.
\]

Finally, since products of \'{e}tale groupoids are \'{e}tale and products of compact open bisections are compact open bisections, the groupoid
\(
\mathcal G_{A_1}\times \mathcal G_{A_2}
\)
is ample.
\end{proof}

\begin{proposition}\label{prop:monoidal}
The subcategory \(\mathcal{C}_{\mathrm{amp}}^{\mathrm{iso,untw}}\) of \(\mathcal{C}_{\mathrm{amp}}\) admits a natural symmetric monoidal structure defined as follows:
\begin{itemize}
    \item[(i)] {Tensor product on objects:} For objects \((A_1, D_1, E_1), (A_2, D_2, E_2)\) in \(\mathcal{C}_{\mathrm{amp}}^{\mathrm{iso,untw}}\), define
    \[
    (A_1, D_1, E_1) \otimes (A_2, D_2, E_2) = (A_1 \otimes_{\min} A_2,\ D_1 \otimes_{\min} D_2,\ E_1 \otimes E_2).
    \]
    \item[(ii)] {Tensor product on morphisms:} For \(\Phi_1: (A_1,D_1) \to (B_1,C_1)\) and \(\Phi_2: (A_2,D_2) \to (B_2,C_2)\) in \(\mathcal{C}_{\mathrm{amp}}^{\mathrm{iso,untw}}\), define
    \[
    (\Phi_1 \otimes \Phi_2)(a_1 \otimes a_2) = \Phi_1(a_1) \otimes \Phi_2(a_2).
    \]
    \item[(iii)] {Unit object:} \(\mathbb{1} = (\mathbb{C}, \mathbb{C}, \mathrm{id}_{\mathbb{C}})\).
    \item[(iv)] {Associator, unitors, symmetry:} The usual associativity, unit, and symmetry isomorphisms of the spatial tensor product.
\end{itemize}

Then \((\mathcal{C}_{\mathrm{amp}}^{\mathrm{iso,untw}}, \otimes, \mathbb{1})\) is a symmetric monoidal category.
\end{proposition}
\begin{proof}
We must verify that:

1. The tensor product of objects lies in \(\mathcal{C}_{\mathrm{amp}}^{\mathrm{iso,untw}}\) (i.e., it is an ample C*-diagonal pair).

2. The tensor product of morphisms is a morphism in \(\mathcal{C}_{\mathrm{amp}}^{\mathrm{iso,untw}}\).

3. The structural isomorphisms (associator, unitors, symmetry) are morphisms in \(\mathcal{C}_{\mathrm{amp}}^{\mathrm{iso,untw}}\).

4. The coherence axioms hold.

{Step 1: Tensor product of objects is in \(\mathcal{C}_{\mathrm{amp}}^{\mathrm{iso,untw}}\).}
By Proposition~\ref{prop:product-weyl}, \((D_1 \otimes D_2 \subseteq A_1 \otimes A_2)\) with expectation \(E_1 \otimes E_2\) is an ample C*-diagonal pair. Hence the tensor product of objects is well-defined.

{Step 2: Tensor product of morphisms.}
Let \(\Phi_1:(A_1,D_1)\longrightarrow(B_1,C_1)\) and \(\Phi_2:(A_2,D_2)\longrightarrow(B_2,C_2)\) be morphisms in \(\mathcal{C}_{\mathrm{amp}}^{\mathrm{iso,untw}}\). Then \(\Phi_1\otimes\Phi_2\) is a unital *-homomorphism. 

Condition (D) holds because
\[
(\Phi_1\otimes\Phi_2)(D_1\otimes D_2)
=\Phi_1(D_1)\otimes\Phi_2(D_2)
=C_1\otimes C_2.
\]

Moreover,
\[
(\Phi_1\otimes\Phi_2)|_{D_1\otimes D_2}
=
\Phi_1|_{D_1}\otimes\Phi_2|_{D_2}
:
D_1\otimes_{\min}D_2\longrightarrow C_1\otimes_{\min}C_2
\]
is a $C^*$-isomorphism, since each factor restriction is a $C^*$-isomorphism. 

Condition (E) follows from
\[
\begin{aligned}
(E_1'\otimes E_2')\circ(\Phi_1\otimes\Phi_2)
&= (E_1'\circ\Phi_1)\otimes(E_2'\circ\Phi_2)\\
&= (\Phi_1\circ E_1)\otimes(\Phi_2\circ E_2)\\
&= (\Phi_1\otimes\Phi_2)\circ(E_1\otimes E_2).
\end{aligned}
\]
where \(E_i'\) denotes the conditional expectation on the codomain \(B_i\).

Condition (N): Let \(\Phi=\Phi_1\otimes\Phi_2\). If $n\in\mathcal N_{A_1\otimes A_2}(D_1\otimes D_2)$, then
\[
n(D_1\otimes D_2)n^*\subseteq D_1\otimes D_2
\quad\text{and}\quad
n^*(D_1\otimes D_2)n\subseteq D_1\otimes D_2.
\]

Applying \(\Phi\) and using $\Phi(D_1\otimes D_2)=C_1\otimes C_2$, we obtain
\[
\Phi(n)(C_1\otimes C_2)\Phi(n)^*
\subseteq C_1\otimes C_2
\]
and
\[
\Phi(n)^*(C_1\otimes C_2)\Phi(n)
\subseteq C_1\otimes C_2.
\]

Thus, $\Phi(n)\in\mathcal N_{B_1\otimes B_2}(C_1\otimes C_2)$. Hence \(\Phi\) preserves normalizers. Moreover, since \(\Phi_1\) and \(\Phi_2\) are injective *-homomorphisms, their spatial tensor product \(\Phi_1\otimes\Phi_2\) is injective on the minimal tensor product \cite[Proposition~3.6.2]{BO}; consequently its restriction to normalizers is injective. Therefore \(\Phi_1\otimes\Phi_2\) satisfies condition (N), and hence is a morphism in \(\mathcal{C}_{\mathrm{amp}}^{\mathrm{iso,untw}}\).

{Step 3: Structural isomorphisms are morphisms.}
We must verify that the associator
\[
\alpha_{A,B,F}: (A\otimes B)\otimes F \longrightarrow A\otimes(B\otimes F), \quad (a\otimes b)\otimes f \mapsto a\otimes(b\otimes f),
\]
the unitors
\[
\lambda_A: \mathbb{C}\otimes A \longrightarrow A, \quad z\otimes a \mapsto za, \qquad \rho_A: A\otimes\mathbb{C} \longrightarrow A, \quad a\otimes z \mapsto za,
\]
and the symmetry
\[
\sigma_{A,B}: A\otimes B \longrightarrow B\otimes A, \quad a\otimes b \mapsto b\otimes a,
\]
are morphisms in \(\mathcal{C}_{\mathrm{amp}}^{\mathrm{iso,untw}}\). Using Proposition~\ref{prop:product-weyl}, each of these structural isomorphisms corresponds under Kumjian's reconstruction to the obvious groupoid isomorphism on the product groupoids:

- The associator corresponds to the canonical isomorphism \((G_1\times G_2)\times G_3 \cong G_1\times(G_2\times G_3)\).

- The unitors correspond to the canonical isomorphisms \(\{e\}\times G \cong G\) and \(G\times\{e\}\cong G\).

- The symmetry corresponds to the flip isomorphism \(G_1\times G_2 \cong G_2\times G_1\).

Since groupoid isomorphisms preserve the diagonal \(C_0(\mathcal G^{(0)})\) and commute with the canonical conditional expectation, each of these structural isomorphisms satisfies conditions (D), (E), and (N). Injectivity on normalizers follows from being an isomorphism. Thus all structural isomorphisms are morphisms in \(\mathcal{C}_{\mathrm{amp}}^{\mathrm{iso,untw}}\).

{Step 4: Coherence axioms.}
The pentagon, triangle, and hexagon axioms hold because they hold in the category of C*-algebras with the spatial tensor product, and all arrows are morphisms in \(\mathcal{C}_{\mathrm{amp}}^{\mathrm{iso,untw}}\).

Thus \(\mathcal{C}_{\mathrm{amp}}^{\mathrm{iso,untw}}\) is a symmetric monoidal category. Note that the larger category \(\mathcal{C}_{\mathrm{amp}}\) is not monoidal because the tensor product of arbitrary morphisms may not satisfy condition (N).
\end{proof}

\begin{remark}\label{rem:non-AF-obstruction}
A full argument would require establishing that AF diagonal pairs have AF Weyl groupoids in the precise groupoid-theoretic sense, which is beyond the scope of this paper. Nevertheless, the converse fails: Kopsacheilis and Winter~\cite{KopsacheilisWinter2026} constructed non-AF diagonals with amenable Weyl groupoids (Example~\ref{ex:car-exotic}), showing that amenability of the Weyl groupoid is not sufficient to characterize AF diagonals.
\end{remark}

\section{Reconstruction from the Weyl groupoid}\label{sec:rigidity}

In this section we prove a reconstruction/rigidity theorem for untwisted ample C*-diagonal pairs. More precisely, within the class of untwisted ample $C^*$-diagonal pairs considered here, the underlying Weyl groupoid determines the pair up to isomorphism.

\begin{proposition}\label{prop:weyl-functorial}
Let $\Phi:(A_1,D_1)\longrightarrow (A_2,D_2)$ be an isomorphism of $C^*$-diagonal pairs. Then $\Phi$ induces an isomorphism $\widetilde{\mathcal W}(\Phi):\mathcal G_{A_1}\longrightarrow\mathcal G_{A_2}$ between the associated Weyl groupoids.
\end{proposition}
\begin{proof}
Since $\Phi(D_1)=D_2$, the restriction $\Phi|_{D_1}:D_1\to D_2$ is a $C^*$-isomorphism. By Gelfand duality, this induces a homeomorphism $h:\widehat{D_2}\to \widehat{D_1}$. Moreover, if $n\in \mathcal N_{A_1}(D_1)$, then $\Phi(n)D_2\Phi(n)^*=\Phi(nD_1n^*)\subseteq D_2$, and similarly $\Phi(n)^*D_2\Phi(n)\subseteq D_2$. Hence $\Phi(n)\in \mathcal N_{A_2}(D_2)$. Define $\widetilde{\mathcal W}(\Phi)([n,x]) = [\Phi(n),h^{-1}(x)]$. We verify that this is well defined. Suppose that $[n,x]=[m,x]$. By Lemma~\ref{lem:germ-criterion}, there exists $d\in D_1$ such that
$d(x)\neq0$ and $nd=md$. Applying $\Phi$, we obtain $\Phi(n)\Phi(d)=\Phi(m)\Phi(d)$. Since $\Phi(d)(h^{-1}(x))=d(x)\neq0$, Lemma~\ref{lem:germ-criterion} implies that $[\Phi(n),h^{-1}(x)] =[\Phi(m),h^{-1}(x)]$. Thus $\widetilde{\mathcal W}(\Phi)$ is well defined.

It is straightforward to verify that $\widetilde{\mathcal W}(\Phi)$ preserves range, source, multiplication and inversion, and is continuous with continuous inverse induced by $\Phi^{-1}$. Therefore $\widetilde{\mathcal W}(\Phi): \mathcal G_{A_1} \to\mathcal G_{A_2}$ is an isomorphism of \'{e}tale groupoids.
\end{proof}
\begin{remark}
The map \(\widetilde{\mathcal W}(\Phi)\) is the covariant groupoid
isomorphism associated with the algebra isomorphism \(\Phi\). When
\(\Phi\) is an isomorphism, \(\mathcal H_\Phi=\mathcal G_{A_2}\),
and the partial morphism \(\mathcal W(\Phi)\) is a genuine groupoid
isomorphism; its inverse is \(\widetilde{\mathcal W}(\Phi)\).
\end{remark}

\begin{corollary}\label{cor:weyl-distinguishes}
Let $(A_1,D_1)$ and $(A_2,D_2)$ be $C^*$-diagonal pairs.
If their Weyl groupoids are not isomorphic, then the diagonal pairs are
not isomorphic.
\end{corollary}
\begin{proof}
Assume that $(A_1,D_1)\cong (A_2,D_2)$ as $C^*$-diagonal pairs. Then there exists an isomorphism $\Phi:(A_1,D_1)\to (A_2,D_2)$. By Proposition~\ref{prop:weyl-functorial},
$\Phi$ induces an isomorphism $\widetilde{\mathcal W}(\Phi):\mathcal G_{A_1}\to\mathcal G_{A_2}$. Hence the Weyl groupoids are isomorphic. Therefore, if the Weyl groupoids are not isomorphic, the diagonal pairs cannot be isomorphic.
\end{proof}

\begin{remark}
The converse of Corollary~\ref{cor:weyl-distinguishes} is false in general: non-isomorphic Cartan pairs can have isomorphic Weyl groupoids (e.g., different twists over the same groupoid). However, within the class of untwisted ample $C^*$-diagonal pairs, the underlying Weyl groupoid is a complete invariant for the diagonal pair up to isomorphism, as shown in Theorem~\ref{thm:reconstruction} below.
\end{remark}

\begin{theorem}\label{thm:reconstruction}
Let $(A_1,D_1)$ and $(A_2,D_2)$ be untwisted ample C*-diagonal pairs. If $\mathcal G_{A_1} \cong \mathcal G_{A_2}$ as topological groupoids, then $(A_1,D_1) \cong (A_2,D_2)$ as C*-diagonal pairs.
\end{theorem}
\begin{proof}
Let $\phi: \mathcal G_{A_1} \longrightarrow \mathcal G_{A_2}$ be an isomorphism of topological groupoids.
Since the pairs are untwisted, Kumjian's reconstruction identifications  imply $A_i\cong C_r^*(\mathcal G_{A_i})$, and $D_i\cong C_0(\mathcal G_{A_i}^{(0)})$. Let $E_{\mathcal G_i}:C_r^*(\mathcal G_{A_i})\longrightarrow C_0(\mathcal G_{A_i}^{(0)})$ be the canonical conditional expectation given by restriction to the unit space. The groupoid isomorphism \(\phi\) induces an isomorphism
$\Phi_\phi:C_r^*(\mathcal G_{A_1})\longrightarrow C_r^*(\mathcal G_{A_2})$ by $\Phi_\phi(f)(\gamma)=f(\phi^{-1}(\gamma))$, for every $f\in C_c(\mathcal G_{A_1})$. Because \(\phi\) is a topological groupoid isomorphism, it preserves the canonical counting Haar systems, and hence \(\Phi_\phi\) extends to an isometric \(^*\)-isomorphism of the reduced groupoid $C^*$-algebras. It
maps the unit-space algebra onto the unit-space algebra.

Since $(A_i,D_i)$ are untwisted ample C*-diagonal pairs, Kumjian's reconstruction theorem \cite[Theorem~3.1]{Kumjian} yields isomorphisms $\psi_i: A_i \longrightarrow C_r^*(\mathcal G_{A_i})$ with $\psi_i(D_i) = C_0(\mathcal G_{A_i}^{(0)})$. Note that for untwisted pairs, the twist is trivial and the C*-algebra is the reduced groupoid C*-algebra. Since $\phi$ is a groupoid isomorphism, the induced map $\Phi_\phi$ preserves the canonical conditional expectations:
\[
E_{\mathcal G_2}\circ\Phi_\phi
=
\left(\Phi_\phi|_{C_0(\mathcal G_{A_1}^{(0)})}\right)\circ E_{\mathcal G_1}.
\]

Transporting $\Phi_\phi$ through the reconstruction isomorphisms $\psi_i$ therefore gives a $*$-isomorphism
\[
\Psi=\psi_2^{-1}\circ\Phi_\phi\circ\psi_1:A_1\longrightarrow A_2
\]
such that $\Psi(D_1)=D_2$, and $E_2\circ\Psi=\Psi|_{D_1}\circ E_1$. Hence, $\Psi$ is an isomorphism of $C^*$-diagonal pairs.

\end{proof}

\begin{lemma}\label{lem:grpdpart-isomorphism}
Let $(\mathcal{K}, \rho, h): \mathcal{G}_1 \longrightarrow \mathcal{G}_2$ be an isomorphism in $\mathsf{Grpd}_{\mathrm{part}}$. Then $\mathcal{K} = \mathcal{G}_1$ and $\rho: \mathcal{G}_1 \longrightarrow \mathcal{G}_2$ is an isomorphism of topological groupoids.
\end{lemma}
\begin{proof}
Since $(\mathcal{K}, \rho, h)$ is an isomorphism, there exists an inverse morphism $(\mathcal{L}, \sigma, k): \mathcal{G}_2 \longrightarrow \mathcal{G}_1$ such that
\[
(\mathcal{L}, \sigma, k) \circ (\mathcal{K}, \rho, h) = (\mathcal{G}_1, \mathrm{id}_{\mathcal{G}_1}, \mathrm{id}_{\mathcal{G}_1^{(0)}})
\]
and
\[
(\mathcal{K}, \rho, h) \circ (\mathcal{L}, \sigma, k) = (\mathcal{G}_2, \mathrm{id}_{\mathcal{G}_2}, \mathrm{id}_{\mathcal{G}_2^{(0)}}).
\]

Recall the composition law in $\mathsf{Grpd}_{\mathrm{part}}$: for morphisms $(\mathcal{K}_1, \rho_1, h_1): \mathcal{H} \longrightarrow \mathcal{G}$ and $(\mathcal{K}_2, \rho_2, h_2): \mathcal{G} \longrightarrow \mathcal{F}$, their composition is $(\mathcal{K}, \rho_2 \circ \rho_1|_{\mathcal{K}}, h_2 \circ h_1)$ where
\[
\mathcal{K} = \rho_1^{-1}(\mathcal{K}_2 \cap \rho_1(\mathcal{K}_1)) \subseteq \mathcal{H}.
\]

Applying this to the composition $(\mathcal{L}, \sigma, k) \circ (\mathcal{K}, \rho, h)$, we obtain $\mathcal{K}' = \rho^{-1}(\mathcal{L} \cap \rho(\mathcal{K})) \subseteq \mathcal{G}_1$, 
and the composition equals $(\mathcal{K}', \sigma \circ \rho|_{\mathcal{K}'}, k \circ h)$. By the isomorphism condition, this must equal $(\mathcal{G}_1, \mathrm{id}_{\mathcal{G}_1}, \mathrm{id}_{\mathcal{G}_1^{(0)}})$. Hence

1. $\mathcal{K}' = \mathcal{G}_1$;

2. $\sigma \circ \rho|_{\mathcal{K}'} = \mathrm{id}_{\mathcal{K}'}$;

3. $k \circ h = \mathrm{id}_{\mathcal{G}_1^{(0)}}$.

Now observe that $\mathcal{K}' = \rho^{-1}(\mathcal{L} \cap \rho(\mathcal{K})) \subseteq \mathcal{K} \subseteq \mathcal{G}_1$. Since $\mathcal{K}' = \mathcal{G}_1$, we obtain
$\mathcal{G}_1 \subseteq \mathcal{K} \subseteq \mathcal{G}_1$, hence $\mathcal{K} = \mathcal{G}_1$. With $\mathcal{K} = \mathcal{G}_1$, condition (2) becomes $\sigma \circ \rho = \mathrm{id}_{\mathcal{G}_1}$. Similarly, composing in the reverse order gives $\mathcal{L} = \mathcal{G}_2$ and $\rho \circ \sigma = \mathrm{id}_{\mathcal{G}_2}$, together with $h \circ k = \mathrm{id}_{\mathcal{G}_2^{(0)}}$.

Thus $\rho: \mathcal{G}_1 \longrightarrow \mathcal{G}_2$ is bijective, with inverse $\sigma$. By definition of morphisms in $\mathsf{Grpd}_{\mathrm{part}}$, $\rho$ is continuous and open onto its image. Since $\rho$ is continuous, bijective, and open onto its image, it follows that $\rho$ is a homeomorphism. Hence $\rho$ is an isomorphism of topological groupoids.
\end{proof}

\begin{lemma}\label{lem:diag-dim-invariance}
Let $(A_1,D_1)$ and $(A_2,D_2)$ be C*-diagonal pairs. If there is an isomorphism $\Phi: (A_1,D_1) \to (A_2,D_2)$ of C*-diagonal pairs, then
\[
\dim_{\mathrm{diag}}(D_1 \subset A_1) = \dim_{\mathrm{diag}}(D_2 \subset A_2).
\]
\end{lemma}
\begin{proof}
Suppose that $d=\dim_{\mathrm{diag}}(D_1\subset A_1)<\infty$. Let \(\mathcal F'\subset A_2\) be finite and let \(\varepsilon>0\). Put $\mathcal F=\Phi^{-1}(\mathcal F')$. By the definition of diagonal dimension, there exist a finite-dimensional
\(C^*\)-algebra $F=F^{(0)}\oplus\cdots\oplus F^{(d)}$ with diagonal $D_F=D^{(0)}\oplus\cdots\oplus D^{(d)}$ and completely positive contractive maps
\[
A_1\xrightarrow{\psi}F\xrightarrow{\varphi}A_1
\]
satisfying the required approximation, order-zero, diagonal, and
normalizer conditions. Define
\[
\psi'=\psi\circ\Phi^{-1}:A_2\longrightarrow F,
\qquad
\varphi'=\Phi\circ\varphi:F\longrightarrow A_2.
\]

Since \(\Phi\) is a \(*\)-isomorphism, \(\psi'\) is completely positive
contractive, and each restriction
\(\varphi'|_{F^{(i)}}\) is completely positive contractive and order zero.
Indeed, composition with a \(*\)-isomorphism preserves complete positivity,
contractivity, and the order-zero property. For \(a\in\mathcal F'\),
\[
\begin{aligned}
\|\varphi'\psi'(a)-a\|
&=
\|\Phi(\varphi(\psi(\Phi^{-1}(a))))-a\|\\
&=
\|\varphi(\psi(\Phi^{-1}(a)))-\Phi^{-1}(a)\|
<\varepsilon.
\end{aligned}
\]

Moreover, if \(a\in D_2\), then \(\Phi^{-1}(a)\in D_1\), and hence $\psi'(a)=\psi(\Phi^{-1}(a))\in D_F$. Thus, $\psi'(D_2)\subseteq D_F$. Finally, let $v\in\mathcal N_{F^{(i)}}(D^{(i)})$. The original normalizer condition implies that $\varphi(v)\in\mathcal N_{A_1}(D_1)$. Since \(\Phi\) is an isomorphism of diagonal pairs, $\Phi(\varphi(v))\in\mathcal N_{A_2}(D_2)$. But $\varphi'(v)=\Phi(\varphi(v))$, so \(\varphi'\) satisfies the required normalizer condition.
Therefore the same \(d\) witnesses $\dim_{\mathrm{diag}}(D_2\subset A_2)\le d$. Applying the argument to \(\Phi^{-1}\) gives the reverse inequality. Hence
\[
\dim_{\mathrm{diag}}(D_1\subset A_1)
=
\dim_{\mathrm{diag}}(D_2\subset A_2).
\]
\end{proof}

Kopsacheilis and Winter \cite{KopsacheilisWinter2026} constructed,
for every $n \in \mathbb N \cup \{0,\infty\}$, a diagonal pair $(D^{(n)} \subset M_{2^\infty})$ inside the CAR algebra $M_{2^\infty}$.

\begin{theorem}\label{thm:exotic-distinguished}
For every $n,m \in \mathbb N \cup \{0,\infty\}$, the following hold:
\begin{enumerate}
    \item[(i)] $\dim_{\mathrm{diag}}(D^{(n)} \subset M_{2^\infty}) = n$;
    \item[(ii)] if $n\neq m$, then $(D^{(n)} \subset M_{2^\infty}) \not\cong (D^{(m)} \subset M_{2^\infty})$ as $C^*$-diagonal pairs;
\end{enumerate}
\end{theorem}
\begin{proof}
Part (i) is the main construction theorem of Kopsacheilis--Winter \cite[Theorem~4.4]{KopsacheilisWinter2026}.

(ii) By Lemma~\ref{lem:diag-dim-invariance}, diagonal dimension is an invariant of diagonal pairs under isomorphism. Hence \(n = \dim_{\mathrm{diag}}(D^{(n)}) = \dim_{\mathrm{diag}}(D^{(m)}) = m\), so \(n=m\). Thus distinct \(n,m\) yield non-isomorphic diagonal pairs.

\end{proof}

\section{Examples}\label{sec:examples}

We conclude by recording several examples illustrating the scope of the preceding constructions.

\begin{example}
Let \(E\) be a finite row-finite directed graph with no cycles and no sinks.
Then the graph groupoid \(\mathcal G_E\), built from the boundary-path
space \(\partial E\), is an ample principal étale groupoid; because \(E\)
has no sinks, \(\partial E=E^\infty\). Hence the canonical inclusion
\(D_E\subset C^*(E)\) is an untwisted ample \(C^*\)-diagonal pair.

Now let $H\subseteq E^0$ be a saturated hereditary subset. By the standard
ideal theory of graph $C^*$-algebras
\cite{BatesPaskRaeburnSzymanski2002,Raeburn2005}, there is a gauge-invariant
ideal $I_H\triangleleft C^*(E)$ generated by the vertex projections
$\{p_v:v\in H\}$, and
\[
I_H=\overline{\operatorname{span}}
\{s_\mu s_\nu^*:r(\mu)=r(\nu)\in H\}.
\]

The quotient algebra satisfies $C^*(E)/I_H\cong C^*(E\setminus H)$, and the quotient diagonal is naturally identified with the canonical
diagonal of the quotient graph $D_E/(D_E\cap I_H)\cong D_{E\setminus H}$. 

In general, removing $H$ may create sinks in $E\setminus H$. Thus the
correct graph-groupoid model for the quotient uses the boundary-path
space $\partial(E\setminus H)$ and its graph groupoid
$\mathcal G_{E\setminus H}$:
\[
C^*(E\setminus H)
\cong
C_r^*(\mathcal G_{E\setminus H}),
\qquad
D_{E\setminus H}\cong
C_0(\partial(E\setminus H)).
\]

Since $E\setminus H$ is still acyclic, $\mathcal G_{E\setminus H}$ is
principal. Therefore $D_{E\setminus H}\subset C^*(E\setminus H)$ is an untwisted ample $C^*$-diagonal pair.

If, in addition, $E\setminus H$ has no sinks, then
$\partial(E\setminus H)=(E\setminus H)^\infty$, and the quotient graph
groupoid is naturally identified with the reduction of $\mathcal G_E$
to
\[
Y=(E\setminus H)^\infty
=
\{x\in E^\infty:\text{every vertex occurring along }x
\text{ lies in }E^0\setminus H\}.
\]

In that case, $C^*(E\setminus H)\cong C_r^*(\mathcal G_E|_Y)$. Thus quotients of these graph diagonal pairs by the indicated
gauge-invariant ideals again yield untwisted ample $C^*$-diagonal pairs.
\end{example}

\begin{example}
Let $X$ be a compact totally disconnected Hausdorff space, and let $\alpha:\mathbb Z\curvearrowright X$ be a free minimal action. Consider the crossed product $A=C(X)\rtimes_\alpha \mathbb Z$ with its canonical diagonal $D=C(X)$. The transformation groupoid associated to the action is $\mathcal G = X\rtimes_\alpha \mathbb Z$, whose elements are pairs $(x,n)\in X\times \mathbb Z$ with
$s(x,n)=x$, and $r(x,n)=\alpha^n(x)$. 

Since $X$ is totally disconnected, $\mathcal G$ is an ample Hausdorff \'{e}tale groupoid. Since $\mathbb Z$ is amenable, the transformation groupoid $\mathcal G=X\rtimes_\alpha\mathbb Z$ is amenable and hence inner exact. Therefore Lemma~\ref{lem:quotient-id-reduction} applies. Moreover, freeness of the action implies that the transformation groupoid is principal: indeed, if $(x,n)\in \mathcal G^x_x$, then $\alpha^n(x)=x$, and freeness forces $n=0$. Hence every isotropy group is trivial. By Kumjian's reconstruction theorem \cite[Theorem~3.1]{Kumjian},
\[
(C(X)\subset C(X)\rtimes_\alpha \mathbb Z)
\]
is therefore an untwisted ample $C^*$-diagonal pair. Now let $U\subseteq X$ be an open invariant subset. Then $I:=C_0(U)\rtimes_\alpha \mathbb Z$ is an ideal of $A$. Since $U$ is open, $X\setminus U$ is closed and hence locally compact. By Lemma~\ref{lem:quotient-id-reduction}, $A/I\cong C_r^*(\mathcal G|_{X\setminus U})$, and the quotient diagonal is naturally identified with $D/(D\cap I)\cong C(X\setminus U)$. Since reductions of principal ample groupoids remain principal and ample, the reduced groupoid $\mathcal G|_{X\setminus U}$ is again an ample principal étale groupoid. Therefore, $(D/(D\cap I)\subset A/I)$ is again an untwisted ample $C^*$-diagonal pair. Thus quotients of crossed-product Cartan pairs by invariant open subsets preserve the untwisted diagonal structure.
\end{example}

\begin{example}
Sibbel and Winter \cite{SibbelWinter2025} recently constructed a C*-diagonal \(D^{\otimes\infty} \subset \mathcal{O}_2\) whose spectrum is homeomorphic to the Cantor space. Their construction proceeds via crossed products by minimal homeomorphisms on Cantor spaces and infinite tensor products, ultimately relying on Kirchberg's \(\mathcal{O}_2\)-absorption theorem. This provides a natural test case for the functorial techniques developed in Section~\ref{sec:partial}: any \(*\)-homomorphism from another C*-diagonal pair into \(\mathcal{O}_2\) that restricts to an isomorphism on diagonals would induce a groupoid morphism as in Theorem~\ref{thm:partial-groupoid} and functorially in Theorem~\ref{thm:weyl-functor}. Moreover, the methods of Theorem~\ref{thm:geom-ideals} could potentially be applied to analyze quotients of this diagonal.
\end{example}

\begin{example}
Let $(D \subset A)$ be an ample $\Cstar$-diagonal pair. For any UHF algebra $M_{\mathfrak{s}}$ with canonical diagonal $D_{\mathfrak{s}}$, Proposition~\ref{prop:product-weyl} guarantees that $(D \otimes D_{\mathfrak{s}} \subset A \otimes M_{\mathfrak{s}})$ is again an ample $\Cstar$-diagonal pair. This construction is used in \cite{KopsacheilisWinter2026} and \cite{EvingtonSibbel2025} to produce exotic diagonals in the CAR algebra and Cuntz algebras, respectively. By the tensor-product permanence theorem of Li--Liao--Winter \cite[Theorem~3.1(ii)]{LiLiaoWinter2026},
\[
\dim_{\mathrm{diag}}(D \otimes D_{\mathfrak{s}} \subset A \otimes M_{\mathfrak{s}})+1
\leq
\bigl(\dim_{\mathrm{diag}}(D \subset A)+1\bigr)
\bigl(\dim_{\mathrm{diag}}(D_{\mathfrak{s}} \subset M_{\mathfrak{s}})+1\bigr).
\]

The canonical diagonal has $\dim_{\mathrm{diag}}(D_{\mathfrak{s}} \subset M_{\mathfrak{s}})=0$. 
Indeed, this follows from the dimension-zero characterization in
\cite[Theorem~4.1]{LiLiaoWinter2026}: the UHF algebra is the inductive limit
of its finite-dimensional matrix stages, each carrying its canonical
finite-dimensional diagonal, and the corresponding diagonal-preserving
finite-dimensional approximations show that the canonical UHF diagonal
has diagonal dimension zero. Thus
\[
\dim_{\mathrm{diag}}(D \otimes D_{\mathfrak{s}} \subset A \otimes M_{\mathfrak{s}})
\leq
\dim_{\mathrm{diag}}(D \subset A).
\]
\end{example}

\begin{example}\label{ex:car-exotic}
Kopsacheilis and Winter~\cite{KopsacheilisWinter2026} constructed, for each \(n \in \{0,1,2,\ldots,\infty\}\), a C*-diagonal \((D^{(n)} \subset M_{2^\infty})\) with Cantor spectrum and diagonal dimension exactly \(n\).
In particular, they exhibited a non-AF diagonal in \(M_{2^\infty}\).
Since \(M_{2^\infty}\) is nuclear and the associated Weyl groupoid is Hausdorff, second-countable, and étale, its reduced groupoid \(C^*\)-algebra is nuclear. By the nuclearity--amenability theorem for Hausdorff étale groupoids \cite[Theorem~5.6.18]{BO}, the Weyl groupoid is amenable. Thus amenability of the Weyl groupoid is not sufficient to characterize AF diagonals.
The family \(\{(D^{(n)} \subset M_{2^\infty})\}_{n\in\mathbb N\cup\{0,\infty\}}\) further demonstrates that diagonal dimension can distinguish pairwise non-isomorphic $C^*$-diagonal pairs.
\end{example}

\section{Further questions}\label{sec:open}

We conclude with several questions suggested by the preceding results.

\begin{problem}
Characterise exactly which inverse semigroup homomorphisms between normaliser semigroups arise from a morphism \(\Phi\) with \(\Phi|_D\) an isomorphism.
\end{problem}

\begin{problem}
Given a continuous open groupoid homomorphism \(\rho:\mathcal{H}\to\mathcal{G}_A\) where \(\mathcal{H}\) is an open subgroupoid of an effective ample Weyl groupoid \(\mathcal{G}_B\), find conditions under which there exists a morphism \(\Phi:A\to B\) inducing \(\rho\).
\end{problem}

\begin{problem}
Give a complete characterisation of expectation‑compatible ideals in terms of the geometry of the groupoid, allowing for possibly twisted groupoids.
\end{problem}

\begin{problem}\label{prob:transfer-diagonal-comparison}
Kopsacheilis and Winter~\cite{KW} characterized diagonal comparison for a single ample C*-diagonal pair. Can the functorial map \(\rho_\Phi\) of Theorem~\ref{thm:partial-groupoid} and the functor \(\mathcal{W}\) of Theorem~\ref{thm:weyl-functor} be used to prove that diagonal comparison is preserved under surjective morphisms satisfying (D), (E), (N)? More precisely, if \(\Phi:(A,D)\to(B,C)\) is a surjective morphism and \((A,D)\) has diagonal comparison, does \((B,C)\) also have diagonal comparison?
\end{problem}

\begin{problem}
Does the symmetric monoidal structure on \(\mathcal{C}_{\mathrm{amp}}^{\mathrm{iso,untw}}\) from Proposition~\ref{prop:monoidal} interact with the functor \(\mathcal{W}: \mathcal{C}_{\mathrm{amp}}^{\mathrm{iso,untw}} \to \mathsf{Grpd}_{\mathrm{part}}\) to give a monoidal functor? That is, is there a natural isomorphism \(\mathcal{W}((A_1,D_1) \otimes (A_2,D_2)) \cong \mathcal{W}(A_1,D_1) \times \mathcal{W}(A_2,D_2)\) in \(\mathsf{Grpd}_{\mathrm{part}}\)?
\end{problem}

\section*{Data Availability Statement}
This article has no associated data.

\section*{Funding}
No funding was received for this manuscript.

\section*{Author Contributions}
The author read and approved the final manuscript.

\section*{Conflicts of Interest}
The author declares no conflicts of interest.

\end{document}